\documentclass[11pt]{amsart}

\usepackage[ansinew]{inputenc}
\usepackage[a4paper,dvips]{geometry}
\usepackage{latexsym}
\usepackage{amsfonts}
\usepackage{amsmath,amssymb}
\usepackage{bbm}
\usepackage[dvipsnames]{xcolor}
\usepackage{tikz}
\usetikzlibrary{decorations.pathreplacing,angles,quotes}

\usepackage{fixme}
\usepackage{graphicx}
\allowdisplaybreaks

\makeatletter
\newcommand{\addresseshere}{%
  \enddoc@text\let\enddoc@text\relax
}
\makeatother

\newtheorem{theorem}{Theorem}
\newtheorem{corollary}{Corollary}

\newtheorem{lemma}{Lemma}
\newtheorem{remark}{Remark}
\newtheorem{definition}{Definition}
\newtheorem{proposition}{Proposition}
\newtheorem{conjecture}{Conjecture}

\newtheorem{question}{Question}

\newcommand{\bv}[1]{\mathbf{#1}}
\DeclareMathOperator{\M}{M}

\newcommand{\NN}{\mathbb{N}}

\newcommand{\FF}{\mathbb{F}}
\newcommand{\PP}{\mathbb{P}}
\newcommand{\eps}{\varepsilon}

\newcommand{\Var}{{\mathbb V}\mathrm{ar}}
\newcommand{\E}{{\mathbb E}}
\newcommand{\cP}{{\mathcal P}}
\newcommand{\cS}{{\mathcal S}}
\newcommand{\cL}{{\mathcal L}}
\newcommand{\cC}{{\mathcal C}}
\newcommand{\cR}{{\mathcal R}}
\newcommand{\cK}{{\mathcal{K}}}
\newcommand{\N}{{\mathbb N}}
\newcommand{\R}{{\mathbb R}}

\newcommand{\bq}{{\mathbf{q}}}
\newcommand{\bOmega}{{\mathbf{\Omega}}}
\newcommand{\bx}{{\mathbf{x}}}
\newcommand{\bu}{{\mathbf{u}}}
\newcommand{\bX}{{\mathbf{X}}}
\newcommand{\by}{{\mathbf{y}}}

\newcommand{\1}{\mathrm{1}}
\newcommand{\dd}{\mathrm{d}}
\DeclareMathOperator{\inter}{int}
\DeclareMathOperator{\bd}{bd}

\newcommand{\cercle}[4]{
\node[circle,inner sep=0,minimum size={2*#2}](a) at (#1) {};
\draw[thick, fill=gray!30] (a.#3) arc (#3:{#3+#4}:#2);
}
\newcommand{\cercleD}[4]{
\node[circle,inner sep=0,minimum size={2*#2}](a) at (#1) {};
\draw[thick, style=dotted, fill=gray!30] (a.#3) arc (#3:{#3+#4}:#2);
}

\title[Discrepancy of stratified samples]{\large Discrepancy of stratified samples\\ from partitions of the unit cube}

\author{Markus Kiderlen}
\address{Aarhus University, Aarhus, Denmark}
\email{kiderlen@math.au.dk}

\author{Florian Pausinger}
\address{Queen's University Belfast, Belfast, United Kingdom.}
\email{f.pausinger@qub.ac.uk}

\date{}

\begin{document}

\keywords{Jittered sampling; Stratified sampling; $L_2$-discrepancy; sets of positive reach}
\subjclass[2010]{ 11K38, 60C05 (primary), and 05A18, 60D99 (secondary)}
\maketitle

%%%%%%%%%%%%%%%%%%%
%	Abstract
%%%%%%%%%%%%%%%%%%%

\begin{abstract}
We extend the notion of jittered sampling to arbitrary partitions and study the discrepancy of the related point sets. Let $\bOmega=(\Omega_1,\ldots,\Omega_N)$ be a partition of $[0,1]^d$ and let the $i$th point in $\cP$ be chosen uniformly in the $i$th set of the partition (and stochastically independent of the other points), $i=1,\ldots,N$. For the study of such sets we introduce the concept of a uniformly distributed triangular array and compare this notion to related notions in the literature.
We prove that the expected ${\cL_p}$-discrepancy, $\E {\cL_p}(\cP_\bOmega)^p$, of a point set $\cP_\bOmega$ generated from any equivolume partition $\bOmega$ is always strictly smaller than the expected ${\cL_p}$-discrepancy of a set of $N$ uniform random samples for $p>1$.
For fixed $N$ we consider classes of stratified samples based on equivolume partitions of the unit cube into convex sets or into sets with a uniform positive lower bound on their reach. It is shown that these classes contain at least one minimizer of the expected    ${\cL_p}$-discrepancy. We illustrate our results with explicit constructions for small $N$.
In addition, we present a family of partitions that seems to improve the expected discrepancy of Monte Carlo sampling by a factor of 2 for every $N$.
%\\[6pt]
%{\bf Last Update:} 27/8/2020
\end{abstract}

%%%%%%%%%%%%%%%%%%%
%	Introduction
%%%%%%%%%%%%%%%%%%%
%\section*{Comments}
%\begin{itemize}
%	\item I have changed the notation for partitions to $\bOmega$ (that is bold!) in order to distinguish between sets $\Omega$ and partitions $\bOmega$. 
%	\item I reformulated the result Example 2 to align it with earlier notation. I added a formal argument why convex equivolume partitions for $N=2$ can be parametrized by one line through $(1/2,1/2)$. In this context: I prefer to remove  $\varphi$ and only talk of the $A$-parametrization in the proof. This is purely cosmetic to help the reader. 
%	\item To Do: Halboffene intervalle sind am Anfang $[a,b[$ aber spaeter $[a,b)$. 
%\end{itemize}

\section{Introduction}
%\begin{itemize}
%\item motivated by recent papers (stefan1, stefan2, doerr?) we would like to take a systematic look at jittered sampling. \textcolor{blue}{[explain jittered sampling][put main three questions of examples files already here and then motivate further stuff with them.]}
%\item we introduce a formal framework to address several questions that came up in (stefan1). in particular, we formalise the notion of a ud family of partitions enabling the systematic study of constructions of partitions similar to the study of ud sequences of points. This allows us to discuss the role that the equivolume property plays for partitions.
%\item Moreover, we refine the Partition principle (theorem 1.2 in stefan1): we know (and show a stronger version of it) that taking paritions improves Monte Carlo. But how to find good partitions? Classical jittered sampling works only for a quadratic number of points.
%\item Finally we address the question whether there is an optimal partition for a given $N$? what are the constraints we need to keep in order to keep the probabilistic flavour?
%\item What about other $L_p$ discrepancies?
%\item \textcolor{red}{[[add also questions from examples file. conjecture: among equivolume classical jittered is best. if we relax equivolume to asymptotically equivolume, we may be able to improve results for a given $N$. see examples for $N=4$.]]}
%\end{itemize}

\subsection{Setting and main questions}
Given a finite set $\cP = \left\{ \bx_1, \dots, \bx_N\right\}$ of $N$ points in $[0,1]^d$ one way to quantify how well-spread these points are, is to calculate the $\cL_p$-discrepancy 
$$ \cL_{p} (\cP) := \left( \int_{[0,1]^d} \left| \frac{\#\left(\cP\cap[0, {\bv x}[\right)}{N} - \big|[0, \bv x[\big| \right|^p \dd\bv x \right)^{1/p}, $$
of $\cP$, in which $1\leq p < \infty$, 
$\#\left(\cP\cap[0, {\bv x}[\right)$ counts the number of indices $1\leq i \leq N$ such that $\bx_i \in [0, {\bv x}[$,
and  $\big|[0, \bv x[\big|$ is the Lebesgue measure of 
$[0,\bv x[ :=\prod_{k=1}^d [0,x_k[$
with $\bv x = (x_1, \ldots, x_d)$; i.e. the $\cL_p$ norm of the so-called discrepancy function.
For an infinite sequence $\cS$ the $\cL_p$-discrepancy $\cL_{p}(\cS_N)$ is the $\cL_p$-discrepancy
of the first $N$ elements, $\cS_N$, of $\cS$.
%The most commonly used discrepancy is the one based on integrated squared deviations, that is, $p=2$. 
Another important irregularity measure is the star-discrepancy defined as
$$
D^* (\cP) :=  \sup_{\bv x\in [0,1]^d} \left | \frac{\#\left(\cP\cap[0, {\bv x}[\right)}{N} - \big|[0, \bv x[\big| \right |.
$$
The $\cL_2$-discrepancy is a well studied and understood measure for the irregularities of point sets. We refer to the book \cite{DP} and the excellent survey \cite{DP2} for further details. In particular, and in contrast to other measures such as the star-discrepancy, it is known how to construct deterministic point sets with the optimal order of magnitude of the $\cL_2$-discrepancy; see \cite{chen1,DP2,DP3}. 
%Our results are concerned with $d=2$ and in that case 
For $d=2$ the optimal order of the $\cL_2$-discrepancy for finite point sets is known to be $\mathcal{O}(\sqrt{\log N}/N)$, which already goes back to a result of Davenport \cite{dave}.
The optimality of these constructions follows from a seminal result of Roth \cite{roth} who derived a general lower bound for the $\cL_2$-discrepancy of arbitrary sets of $N$ points in $[0,1]^d$; see e.g. \cite[Theorem 3.20]{DP}.
%Roth showed that the $\cL_2$-discrepancy of any set of $N$ points always satisfies
%$$  c_d \frac{(\log N)^{(d-1)/2}}{N},$$
%in which $c_d$ is a constant that only depends on the dimension 
While deterministic point sets with small discrepancy are widely used in the context of numerical integration, simulations of different real world phenomena may require an element of randomness. The expected discrepancy of a set $\cP_N$ of $N$ i.i.d.~uniform random points in $[0,1]^d$ is of order $\mathcal{O}(1/\sqrt{N})$ and as such independent of the dimension.
For two-dimensional point sets of $N$ i.i.d.~uniform random points, we thus also have an expected discrepancy of order $\mathcal{O}(1/\sqrt{N})$ similar to the two-dimensional regular grid (whose discrepancy is known to get worse as the dimension increases).

Randomized quasi-Monte Carlo (RQMC) sampling is a popular method to randomize deterministic point sets; see \cite{ecu1} for an excellent introduction. Clever constructions of deterministic point sets, so called quasi-Monte Carlo (QMC) sampling can significantly improve the asymptotic order of integration errors when compared to classical Monte Carlo sampling. RQMC basically takes a deterministic QMC point set as an input and uses a randomisation technique (e.g. a random shift modulo 1 or a so-called digital shift) to generate a new point set, which can be shown to have improved uniform distribution properties compared to Monte Carlo samples, while still enjoying the advantages of being `random' in theoretical analysis; see \cite{cranley,ecu2,haber,owen1,owen2,owen3}.

\begin{center}
\begin{figure}[h!]
\centering
\begin{tikzpicture}[scale=0.65]
\draw[step=1cm,gray,very thin] (0,0) grid (5,5);
\draw[step=1cm,gray,very thin] (8,0) grid (13,5);
\draw[gray,very thin] (16,0) -- (21,0) -- (21,5) -- (16,5)--(16,0);
%left point set
\node at (0,0) {$\bullet$}; 
\node at (0,1) {$\bullet$}; 
\node at (0,2) {$\bullet$}; 
\node at (0,3) {$\bullet$}; 
\node at (0,4) {$\bullet$}; 

\node at (1,0) {$\bullet$}; 
\node at (1,1) {$\bullet$}; 
\node at (1,2) {$\bullet$}; 
\node at (1,3) {$\bullet$}; 
\node at (1,4) {$\bullet$}; 

\node at (2,0) {$\bullet$}; 
\node at (2,1) {$\bullet$}; 
\node at (2,2) {$\bullet$}; 
\node at (2,3) {$\bullet$}; 
\node at (2,4) {$\bullet$}; 

\node at (3,0) {$\bullet$}; 
\node at (3,1) {$\bullet$}; 
\node at (3,2) {$\bullet$}; 
\node at (3,3) {$\bullet$}; 
\node at (3,4) {$\bullet$}; 

\node at (4,0) {$\bullet$}; 
\node at (4,1) {$\bullet$}; 
\node at (4,2) {$\bullet$}; 
\node at (4,3) {$\bullet$}; 
\node at (4,4) {$\bullet$}; 

%middle point set
\node at (8.31,0.23) {$\bullet$}; 
\node at (8.7,1.5) {$\bullet$}; 
\node at (8.2,2.7) {$\bullet$}; 
\node at (8.3,3.34) {$\bullet$}; 
\node at (8.9,4.29) {$\bullet$}; 

\node at (9.3,0.74) {$\bullet$}; 
\node at (9.55,1.56) {$\bullet$}; 
\node at (9.1,2.83) {$\bullet$}; 
\node at (9.65,3.42) {$\bullet$}; 
\node at (9.45,4.43) {$\bullet$}; 

\node at (10.34,0.41) {$\bullet$}; 
\node at (10.67,1.66) {$\bullet$}; 
\node at (10.52,2.42) {$\bullet$}; 
\node at (10.67,3.37) {$\bullet$}; 
\node at (10.52,4.12) {$\bullet$}; 

\node at (11.5,0.82) {$\bullet$}; 
\node at (11.7,1.58) {$\bullet$}; 
\node at (11.3,2.21) {$\bullet$}; 
\node at (11.9,3.14) {$\bullet$}; 
\node at (11.7,4.72) {$\bullet$}; 

\node at (12.23,0.63) {$\bullet$}; 
\node at (12.56,1.61) {$\bullet$}; 
\node at (12.89,2.25) {$\bullet$}; 
\node at (12.17,3.57) {$\bullet$}; 
\node at (12.67,4.19) {$\bullet$}; 

%right point set
\node at (17.2039, 0.0104) {$\bullet$}; 
\node at (19.2219, 3.8681) {$\bullet$}; 
\node at (19.1364, 3.0262) {$\bullet$}; 
\node at (19.904, 1.9185) {$\bullet$}; 
\node at (16.5854, 2.6663) {$\bullet$}; 

\node at (18.7159, 2.4971) {$\bullet$}; 
\node at (19.8897, 2.0048) {$\bullet$}; 
\node at (20.8751, 0.3469) {$\bullet$}; 
\node at (18.674, 4.25) {$\bullet$}; 
\node at (20.1753, 4.479) {$\bullet$}; 

\node at (16.8948, 3.299) {$\bullet$}; 
\node at (19.6393, 3.4874) {$\bullet$}; 
\node at (18.7296, 3.1135) {$\bullet$}; 
\node at (17.2053, 3.8576) {$\bullet$}; 
\node at (20.8653, 3.7005) {$\bullet$}; 

\node at (18.2948, 1.2544) {$\bullet$}; 
\node at (18.9047, 0.153) {$\bullet$}; 
\node at (16.0157, 1.0807) {$\bullet$}; 
\node at (18.3935, 0.1798) {$\bullet$}; 
\node at (19.9259, 3.6681) {$\bullet$}; 

\node at (16.8825, 0.7463) {$\bullet$}; 
\node at (16.3982, 0.0242) {$\bullet$}; 
\node at (16.1357, 1.9686) {$\bullet$}; 
\node at (19.0286, 3.2103) {$\bullet$}; 
\node at (17.0463, 0.3119) {$\bullet$}; 
\end{tikzpicture}
\caption{The regular grid, a point set obtained by classical jittered sampling and a set of i.i.d uniform random points with $N=25$.} \label{fig:def}
\end{figure}
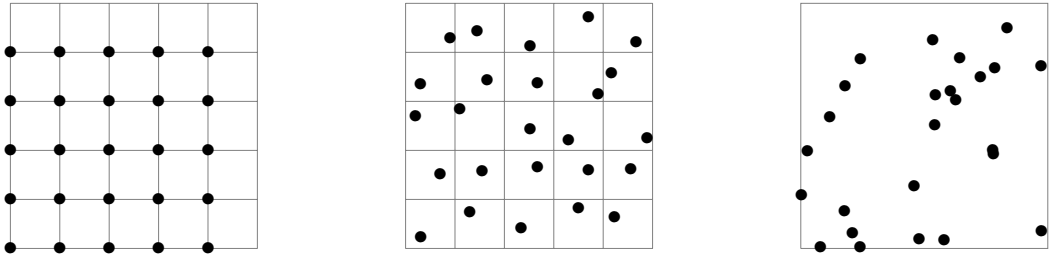
\end{center}

The starting point for our work, can also be considered as a basic RQMC technique which was already discussed in \cite{haber} in a slightly more general form.
Classical \emph{jittered sampling} for $N=m^d$ combines the simplicity of grids with uniform random sampling by partitioning $[0,1]^d$ into $m^d$ axis-aligned congruent cubes and placing a random point inside each of them; see Fig. \ref{fig:def}. Jittered sampling is sometimes referred to as `stratified sampling' in the literature, but we will use the term `stratified sampling' in a more broad sense as outlined below.
Motivated by recent progress \cite{stefan1, stefan2} the aim of this paper is to take a systematic look at 
jittered sampling and its extension based on more general partitions $\bOmega=(\Omega_1,\ldots,\Omega_N)$ of $[0,1]^d$. 
%To avoid confusion, we fix the nomenclature for our paper as follows.
We consider \emph{stratified sampling}, where $[0,1]^d$ is partitioned into $N$ subsets $\Omega_1,\ldots,\Omega_N$  and the $i$th point in $\cP$ is chosen uniformly in the $i$th set of the partition (and stochastically independent of the other points), $i=1,\ldots,N$. 
If $N=m^d$ and the partition consists of the above mentioned axis-aligned congruent cubes, we obtain \emph{jittered sampling} as a special case. Besides results for fixed $N$, we are also interested in the behavior of stratified samples derived from sequences of partitions when $N$ becomes large. 

At this point, we would like to emphasize that sequences of partitions that can be used in stratified sampling are more general than those in Kakutani's splitting procedure and its variants \cite{kakutani,pyke,volcic}. 
Apart from the obvious difference that these procedures restrict considerations to $d=1$, the partitions in the present paper need not be nested. This means that the partition in step $N+1$ is not necessarily obtained as a refinement of the partition in step $N$; see also the discussion in  Appendix A.

It was shown in \cite{stefan1} that the asymptotic order of the star-discrepancy of a point set obtained from  jittered sampling is $\mathcal{O}(N^{-\frac{1}{2}-\frac{1}{2d}})$. Thus, taking partitions can significantly improve the expected discrepancy of (random) point sets in \emph{small} dimensions $d\geq 2$. 
We are interested in the following main questions:
\begin{enumerate}
\item In which sense are sequences of stratified sample points \emph{uniformly distributed} as their number $N$ increases? What is the connection to similar notions for partitions in the literature? 	
\item Does stratified sampling yield smaller or larger mean discrepancy than Monte Carlo sampling with $N$ 
i.i.d.~uniform random points? Are there  discrepancy notions and assumptions assuring that stratified sampling is strictly better? 	
\item Is there a 'best' partition for a given $N$ in terms of a chosen mean discrepancy? 
\item Is there a simple family of partitions $\{\bOmega^{(N)}\}_{N\geq 1}$ that gives reasonable results for all $N$ and not just for square numbers of points as in the case of classical jittered sampling?
\item Can we improve classical jittered sampling with stratified sampling?
\end{enumerate}
Section \ref{sect:2}  presents our answers to the above together with open questions for future research. In Section \ref{sect:3} we prove our main theoretical results and illustrate them with examples. Section \ref{sect:4} introduces and explores an infinite family of partitions and contains more examples as well as numerical results.

\subsection{Stratifications and the star discrepancy} \label{sect:extra}
By the celebrated result of Heinrich, Novak, Wasilkowski \& Wozniakowski \cite{hnww} there exists
a set of $N$ points in $[0,1]^d$ with
\begin{equation}
	\label{HNWW} 
	D^*(\mathcal{P}) \leq c \sqrt{\frac{d}{N}} \qquad \mbox{for some universal constant}~c.
\end{equation} 
Aistleitner \cite{aisti}, using a result of Gnewuch \cite{gnew}, has shown that one can take $c=10$. Doerr \cite{dorr} has shown 
$\mathbb{E}  D^*(\mathcal{P}) \gtrsim \sqrt{d/N}$ for point sets of $N$ i.i.d. uniformly random points indicating that this is the correct order of magnitude. 
See \cite{gnewuch} and references therein for the most up-do-date history of improvements on the constant $c$; the currently smallest value is $c=2.4968$ derived in \cite[Corollary 3.6]{gnewuch}
in which it is also shown that \eqref{HNWW} with $c=3$ holds with very high probability when $\cP$ is a set of $N$ i.i.d. uniformly random points.

Since the best known construction is purely probabilistic, it is natural to ask whether we can improve upon these upper bounds using stratification. Indeed, Aistleitner muses in \cite{aisti} that a \emph{thought-out partition} may improve the upper bound. 
Our strong partition principle (Theorem \ref{thm1}) shows that the mean $\cL_{p}$-discrepancy of stratified sets is strictly smaller than the mean $\cL_{p}$-discrepancy of $N$ i.i.d uniform random points and this could lead to a similar result for the star discrepancy using the technique from \cite{hnww}; see also \cite{nied}. 
However, the main obstacle in this context is that in order to see the stratification effect in large dimensions, one needs to subdivide the unit cube into (exponentially in d) many subsets and hence one faces a seemingly unavoidable difficulty if one wishes for a result for \emph{small} $N$ in \emph{large} dimension. This is also underlined by the discussion on classical jittered sampling in \cite[Section 6]{stefan1}, in which it is detailed why jittered sampling gains in effectiveness over purely random points only around $N \sim d^d$.
We believe that stratifications are most useful in \emph{small} dimensions in which the stratification effect is significant. 

%It is an interesting problem for future research to study the exact range of $d$ in which this stratification effect is most significant. 
\begin{question}
In which range of $d$ is the stratification effect most significant?
\end{question}

We will illustrate the potential of stratifications in the context of star discrepancy with numerical experiments. For fixed (and small) $d$ we expect that it is possible to improve the constant for a uniformly at random scheme with a stratified scheme similar to the case of the $\cL_{p}$-discrepancy. For $d=2,3,5$ we numerically obtain improvements for the families of partitions studied in this paper; see Table \ref{table:discrepancy}. 
%We expect similar results for $d>2$ for appropriate generalisations of our constructions. However, due to its technical difficulty, this is also left as a project for future research.

%Furthermore, we refer to papers of Gnewuch \cite{gne} and Hinrichs \cite{hin} (see also \cite{sst}) as well as the recent comprehensive monographs of
%Novak \& Wozniakowski \cite{nov1, nov2, nov3}.

%%%%%%%%%%%%%%%%%%%%%%%%%%%%%%%%%%%%%%%%%%%

\section{Results} \label{sect:2}
\subsection{Stratified sampling and uniform distribution.}\label{sec2.1}
Let $d\ge 1$ be given. 
We consider partitions $\bOmega=\{ \Omega_1, \ldots,$ $ \Omega_N\}$ of the unit cube in $\R^d$ into $N$ Lebesgue-measurable sets, i.e.
 \begin{align}\label{eq:cover}
[0,1]^d=\bigcup_{i=1}^N \Omega_i,
\end{align}
and the sets do not overlap in the $L^1$-sense, so $\Omega_i\cap \Omega_j$ is a Lebesgue-null set for all $i\ne j$ in $\{1,\ldots,N\}$. It should be emphasized that we prefer this condition to the stronger one requiring pairwise disjoint sets, as we later want to work with closed sets. When the sets $\Omega_i$ and $\Omega_j$ are convex, then $|\Omega_i\cap \Omega_j|=0$ is equivalent to saying that $\Omega_i$ and $\Omega_j$ do not have any interior points in common.
For the moment, we pose no other geometric conditions on the partition, in particular the sets $\Omega_i$ need not be connected.  
\begin{figure}[h!]
	\centering
	\begin{tikzpicture}[scale=0.6]
		%%%%%%%%%%%%%%%%%%%%%%%
		\draw[thick,fill=gray!25] (0,0) -- (0,2) -- (4.5,2) -- (4.5,0) -- (0,0);
		\draw[thick] (0,0) -- (7,0)-- (7,7) -- (0,7)-- (0,0);
		%\draw (0,0) -- (5,5);
		%\draw (0,2) -- (3,2) -- (3,0);
		
		\draw[thick] (1,7) -- (1,0);
		\draw[thick] (2,7) -- (2,0);
		\draw[thick] (3,7) -- (3,0);
		\draw[thick] (4,7) -- (4,0);
		\draw[thick] (5,7) -- (5,0);
		\draw[thick] (6,7) -- (6,0);
		
		%\draw[thick] (1,5) -- (5,1);
		%\draw[thick] (2,5) -- (5,2);
		%\draw[thick] (3,5) -- (5,3);
		%\draw[thick] (4,5) -- (5,4);
		
		\node at (0,2) {$\bullet$}; 
		\node at (4.5,2) {$\bullet$}; 
		\node at (4.5,0) {$\bullet$}; 
		%\node at (2.5,2.5) {$\bullet$}; 
		
		\node at (-0.5,2) {\footnotesize $y$}; 
		\node at (4.5,-0.5) {\footnotesize $x$}; 
		%\node at (2,-0.3) {\footnotesize$m_2$}; 
		%\node at (2.8,2.5) {\footnotesize$v_5$}; 

		%%%%%%%%%%%%%%%%%%%%%%%%%%%%%%
	\end{tikzpicture}
	\qquad
	\qquad
	\begin{tikzpicture}[scale=1.4]
		%%%%%%%%%%%%%%%%%%%%%%%
		
		\draw[thick] (0,0) -- (3,0)-- (3,3) -- (0,3)-- (0,0);
		\draw[dashed] (0,0) -- (3,3);
		
		\draw[thick] (0,1.73) -- (1.73,0);
		\draw[thick] (0,2.45) -- (2.45,0);
		\draw[thick] (0,3) -- (3,0);
		\draw[thick] (3-1.73,3) -- (3,3-1.73);
		\draw[thick] (3-2.45,3) -- (3,3-2.45);

		\node at (1.73,-0.2) {\textcolor{white}{\footnotesize $m_1$}};

		%%%%%%%%%%%%%%%%%%%%%%%%%%%%%%
	\end{tikzpicture}
	\caption{Examples of simple partitions of the unit cube in $\R^2$. Left: A partition of $[0,1]^2$ into $N=7$ vertical strips. Right: Illustration of the partition $\bOmega^{(6)}_{\ast}$ consisting of  $N=6$ equivolume slices that are orthogonal to the diagonal.} \label{fig:simplePartition}
\end{figure}
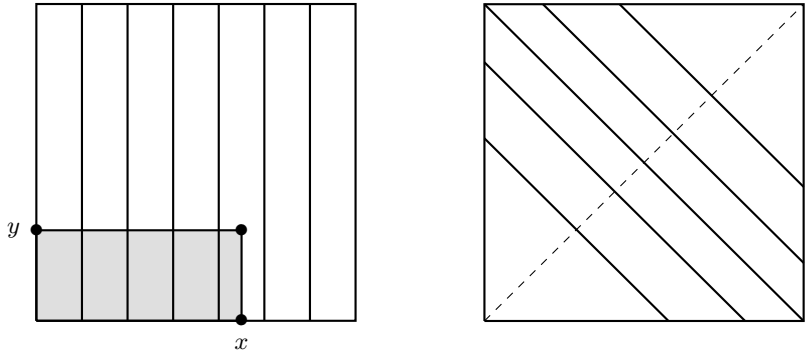

Any partition gives rise to an $N$-element \emph{stratified sample} $\cP_{\bOmega}$ of $N$ random points derived from the partition by picking a random uniform (random w.r.t.~the normalized Lebesgue measure) point from each $\Omega_i$ in a stochastically independent manner. In contrast to stratified sampling in classical sampling theory (see, e.g.~\cite{thompson}), we sample only one point in each of the strata.  
 As ground model for comparison we often will consider the set of \emph{Monte Carlo}  samples $\cP_N$ consisting of $N$ i.i.d.~(independent and identically distributed) uniform random points in the unit cube.

For a set $\cP\subset [0,1]^d$ of $N\in \N$ points in the unit cube let
\begin{align}\label{eq:Zx}
Z_\bx(\cP)=\frac{\#\big({\mathcal P}\cap [0,\bx[\,\big)}{N}
\end{align} 
be the proportion of points falling in a test cube $[0,\bx[$ with $\bx\in [0,1]^d$. 
For a Monte Carlo sample,  $Z_\bx(\cP_N)$ is a random variable with mean $|[0,\bx[|$, but $Z_\bx(\cP_\bOmega)$ need not be unbiased for $|[0,\bx[|$ when $\bOmega$ is a partition of the unit cube. 
We show in Proposition \ref{prop:1} below  that $Z_\bx(\cP_\bOmega)$ has mean $|[0,\bx[|$ for all $\bx\in [0,1]^d$
if and only if the partition is \emph{equivolume}, that is, if $|\Omega_1|=\cdots=|\Omega_N|$. 
This corresponds to the concept of self-weighting stratifications in classical sampling of finite populations: as the samples in all strata are equally large (just one point per stratum), the strata must be equal in size. 
The assumption of equivolume partitions is often convenient, as it allows us to interpret the mean of $\cL_p^p(\cP_{\bOmega})$ as an integrated centered $p$th mean; see Equations \eqref{eq:pthcnteredmean}  and \eqref{eq:MC}, below. Two examples of equivolume partitions for $d=2$ are illustrated in Fig. \ref{fig:simplePartition}.

Now let $\{\bOmega^{(N)}\}_{N\geq 1}$ be a sequence of finite partitions of the unit cube and let $\cP_{\bOmega^{(N)}}=\{\bX_1^{(N)},\ldots,\bX_N^{(N)}\}$ be the stratified sample associated to the $N$th partition. Note that we use capital letters whenever points are random. 
The fact that partitions for different $N$ need not be related to one another implies that the set of all sampling points forms a triangular array, and we are thus led to define a uniform distribution property for those; see also \cite[Section 3]{persi}.

\begin{definition}\label{def1}
	A triangular array $\widehat{\bx}=\big(\bx_1^{(N)},\ldots,\bx_N^{(N)}\big)_{N\in \NN}$ with points in $[0,1]^d$ is said to be \emph{uniformly distributed}, if for every cube $[{\bf x}, {\bf y}[ \subset [0,1[^d$ we have
	\begin{align}
\label{eq:Def1}
	\lim_{N \rightarrow \infty}  \frac{\#\left(\{\bx_1^{(N)},\ldots,\bx_N^{(N)}\}\cap[\bx, \by[\right)}{N}  = \big|[\bf x, \bf y[\big|. 
	\end{align}
	%
	%We say that a triangular array  $\widehat{\bX}=\big(\bX_1^{(N)},\ldots,\bX_N^{(N)}\big)_{N\in \NN}$  with random vectors in $[0,1]^d$ is said to be \emph{uniformly distributed}, if it is uniformly distributed almost surely. 
\end{definition}

A sequence $(\bx_i)$  in the unit cube is uniformly distributed in the usual sense, if and only if the 
triangular array $(\bx_1,\ldots,\bx_N)_{N\in\NN}$ is uniformly distributed in the sense of Definition \ref{def1}. Hence, this definition generalizes the usual one. As in the classical case, uniform distribution of triangular arrays is equivalent to the weak convergence of the sequence of 'empirical distributions', where the $N$th of those distributions sits on the points $\bx_1^{(N)},\ldots,\bx_N^{(N)}$ giving equal mass to each of them. In other words,  
$(1/N) \sum_{i=1}^N f(\bx_i^{(N)})\to \int_{[0,1]^d}f(\bx)\dd\bx$, as $N\to\infty$, for all continuous functions $f$ on the unit cube.  

Proposition \ref{prop:2} in Appendix A characterizes partitions leading a.s.~to uniformly distributed stratified samples using the strong law of large numbers for triangular arrays. 
	The most important implication of Proposition \ref{prop:2} is  that stratification sequences of \emph{equivolume} partitions are uniformly distributed. This is one reason why our theoretical results are based on equivolume partitions. 
	Appendix A also discusses how Definition \ref{def1} relates to similar concepts in the existing literature.

\subsection{The strong partition principle for stratified sampling.}\label{sec2.2}
Discrepancy measures can be used to compare sets of sampling points. In the case of a set $\cP$ of random sampling points the \emph{mean $\cL_p$-discrepancy} $\E\cL_p^p(\cP)$ is often employed, where $\E$ denotes the probabilistic expectation. One should correctly call  $\E\cL_p^p(\cP)$ the `mean $p$th power $\cL_p$-discrepancy', but we prefer the shorter, slightly misleading form for breviety. 

Certainly, a stratified sample need not be better than a Monte Carlo sample. Consider for instance a partition $\bOmega$ with $N$ sets in 
$[0,1]^2$ where the $N-1$ partitioning sets $\Omega_1,\ldots, \Omega_{N-1}$  are all subsets of $[\delta,1]^2$ with some $\delta\in ]0,1[$. Then the mean $\cL_2$-discrepancy satisfies
\begin{align*}
\E\cL_2(\cP_{\bOmega})^2&\ge\E \int_{[0,\delta]^2}\left( \frac{1_{[0,\bx[}(X^{(N)}_N)}N-|[0,\bx[|\right)^2\dd \bx
\\&=
 \frac{\delta^4}{4N}+\frac{\delta^6}{9}\left(1-\frac2N\right)\ge \frac{\delta^4}{4N}>
 \E {\cL_2}(\cP_N)^2, 
\end{align*}
for all $\delta>(5/9)^{1/4}\approx 0.86$ and $N\ge2$, where the last inequality uses \eqref{eq:MCL2}. 

In contrast to this, stratified samples from \emph{equivolume} partitions are never worse than Monte Carlo samples in terms of the mean $\cL_2$-discrepancy according to the Partition Principle in \cite[Theorem 1.2]{stefan1}.
We strengthen this result in two directions showing firstly that stratified samples from {equivolume} partitions are \emph{strictly} better, and secondly that $\cL_2$-discrepancy can be replaced by 
$\cL_p$-discrepancy with arbitrary $p>1$. 
The main ingredient of our proof is a result by Hoeffding \cite{hoeffding} stating that 
among all Poisson-binomial distributions with given mean, the classical binomial distribution is the most spread-out. 
\begin{theorem} (Strong Partition Principle) \label{thm1}
For any equivolume partition $\bOmega$ of $[0,1]^d$ with $N\ge 2$ sets we have 
\begin{equation}
    \label{eq:Florian} 
\E {\cL_p}(\cP_\bOmega)^p< \E {\cL_p}(\cP_N)^p
\end{equation}
for all $p>1$. 
\end{theorem}
The proof of this theorem will be given in  Section \ref{sec3.2}. One can understand \eqref{eq:Florian} as a continuous analog and extension to the statement in finite population sampling theory that self-weighted stratified sampling is always better (in terms of variance) than simple random sampling, both taken with replacement.

For illustration, we consider the sequence $(\bOmega^{(N)}_{\mathrm{vert}})$ of vertical strip partitions; see  Fig.~\ref{fig:simplePartition} (left), but generalized to the $d$-dimensional case. Direct calculation confirms
\begin{align*}
%\E {\cL_2}(\cP_{\bOmega^{(N)}_{\mathrm{vert}}})^2 <\E {\cL_2}(\cP_N)^2
\E {\cL_2}(\cP_{\mathrm{vert}})^2 <\E {\cL_2}(\cP_N)^2
\end{align*}
for all $N\ge 2$. Both sampling schemes have the same asymptotic order $1/N$, but stratified sampling has a better leading constant; see Section \ref{sec3.2} for details.

\subsection{Partitions with best average discrepancy.}\label{subsect:2.3}
%\textcolor{red}{[[discussion why equivolume. we want to avoid degenerate cases. hence we need an additional condition and equivolume seems reasonable since we ensure ud asymptotically.]]}
%\textcolor{red}{[[Add literature on optimal point sets for L2 discrepancy! start from survey of pillichshammer. needs also to go into introduction!]]}
For a given $N\ge 2$, it is an open problem to assure the existence of  a partition 
whose associated stratified sample has lowest mean discrepancy among all partitions consisting of  $N$ sets. 
We will show such existence results for certain \emph{equivolume} partitions. These results are all based on the topological standard argument that a continuous function attains its minimum on a compact set. This requires the choice of a topology on the family $\cC$ of compact subsets of $[0,1]^d$. We have chosen the well-established Hausdorff metric, as $\cC$ is compact in its generated topology. However, both, the extension of this compactness to partitions and the continuity claim of  $\E {\cL_p}(\cP_{\bOmega})^p$ as a function of the partitioning sets, require the continuity of the volume functional. We will assure this by assuming certain regularity conditions. More precisely, we assume that there is $r>0$ such that the sets $\Omega_1,\ldots,\Omega_N\subset [0,1]^d$ of the partition have \emph{reach} at least $r$, meaning that for any point $\bx$ with distance less than $r$ from $\Omega_i$ there is a \emph{unique} closest point to $\bx$ in $\Omega_i$, $i=1,\ldots,N$; see Fig. \ref{fig:reach}. The class of such sets is very general and contains for instance all compact convex sets in $[0,1]^d$ (as the reach of a closed convex set is infinity). It also contains any given set whose boundary is a piecewise $C^2$-curve such that its finitely many vertices are 'convex', provided that $r>0$ is chosen small enough. Let $\mathfrak{P}_N(r)$ be the class of all \emph{equivolume} partitions consisting of sets with reach at least $r>0$.

Also smaller classes of partitions can be treated. We name here the class $\mathfrak{P}_N^{\mathrm{conv}}$  of \emph{equivolume} convex partitions, which might be relevant for applications, as all sets constituting a convex partition of $[0,1]^d$ are actually convex polytopes, with the total number of vertices being uniformly bounded when $N$ is given. Hence, convex partitions can be described using finitely many parameters, and thus they are, at least in principle, computationally tractable. 
The following main result states the existence of equivolume  partitions yielding the best  mean discrepancy $\E\Delta$ from stratification under regularity. Note that the assumptions on the function $\Delta$, which is some given measure of discrepancy, are very weak.

\begin{theorem}\label{thm:existenceReach}
Let $r>0$ and $N\in \NN$ be given and assume that $\Delta:([0,1]^d)^N\to \R$ is measurable and bounded. Then there exists (at least) one partition $\bOmega^* \in \mathfrak{P}_N(r)$ such that the corresponding stratified sample $\cP_{\bOmega^*}$ minimizes the mean $\Delta$-discrepancy on $\mathfrak{P}_N(r)$; i.e.
		$$
		\min_{\bOmega \in \mathfrak{P}_N(r)} \E \Delta(\cP_\bOmega) = \E \Delta(\cP_{\bOmega^*}).
		$$
	A corresponding statement holds true for $\mathfrak{P}_N^{\mathrm{conv}}$.
\end{theorem}
The standard notions of discrepancy satisfy the assumptions in the above theorem; see the end of Section \ref{subsect:3.3} for details.

\begin{corollary}\label{thm:existenceConvex}
	Let $r>0$, $1\le p<\infty$, and $N\in \NN$ be given. Then there are partitions $\bOmega^p \in \mathfrak{P}_N(r)$ 
		and $\bOmega^* \in \mathfrak{P}_N(r)$ of $[0,1]^d$ such that 
		\[
		\min_{\bOmega \in \mathfrak{P}_N(r)} \E {\cL_p}(\cP_\bOmega)^p= \E {\cL_p}(\cP_{\bOmega^p})^p,
		\]
		and 	
		\[
		\min_{\bOmega \in \mathfrak{P}_N(r)} \E D^* (\cP_\bOmega) = \E D^*(\cP_{\bOmega^*}),
		\]
		respectively.
	
	Corresponding statements hold true for $\mathfrak{P}_N^{\mathrm{conv}}$.
\end{corollary}

For illustration we will determine the optimal convex partition of $[0,1]^2$ for $N=2$ in Subsection \ref{sec:equivolConvexN=2}.

Theorem \ref{thm:existenceReach} and its proof  in Section \ref{subsect:3.3} indicate that the properties of the discrepancy are only of marignal importance while the regularity assumptions on the partitioning sets are crucial in order to show the continuity and compactness. Generally speaking, for such existence statements to hold, we expect that these regularity conditions are not needed, that is, we expect the existence of an equivolume partition of $[0,1]^d$ minimizing a given (rather general) measure of discrepancy. The reasoning for this conjecture is the surmise that minimizers are typically consisting of regular sets. To illustrate this point consider  the simple case $d=1$, $N=2$. Clearly, the class of equivolume partitions contains very complicated 
	pairs of sets, such as fractals, and it is certainly not closed in the Hausdorff-metric nor in the induced $\cL_1$-metric for indicator functions. But Corollary \ref{CORnew} in Section \ref{sec3.2}, shows that the unique
	mean-$\cL_2$-disrepancy minimizing equivolume partition in this case is simply $\{[0,1/2],[1/2,1]\}$ 
	(up to sets of measure zero), which consists of very regular sets.

We cannot claim that the equivolume assumption is needed, but it is crucial for our approach, as it avoids that sets in partition sequences shrink to lower-dimensional sets. Existence statements without the equivolume assumption, though very interesting, would require thus substantially different techniques and are beyond the scope of the present paper.

\begin{figure}[h!]
\centering
\scalebox{0.7}{
\begin{tikzpicture}
\coordinate (center) at (0,0);
\coordinate (center1) at (3,0);
\coordinate (center11) at (1.5,1.82);
\cercle{center1}{2cm}{136}{-272}
\cercle{center}{2cm}{44}{272}
\cercleD{center}{2cm}{44}{-88}
\cercleD{center1}{2cm}{136}{88}

%\cercleR{center11}{0.35cm}{0}{360}

\draw [thick, fill=gray!30]  plot[smooth, tension=0.8, ] coordinates { (7,1.5) (7.5, 1.4) (9,1) (11,1.5) (10.7,0)  (11,-1.5)  (9,-0.8) (7,-1.5) (7,0) (6.7,1.2) (7,1.5)};

%%%%%%%%%%%%%  OLD  %%%%%%%%%%%%%%
%\coordinate (center2a) at (7,1.5);
%\coordinate (center2b) at (9,1.5);
%\coordinate (center2c) at (11,1.5);

%\coordinate (center3a) at (7,-1.5);
%\coordinate (center3b) at (9,-1.5);
%\coordinate (center3c) at (11,-1.5);

%\coordinate (center5) at (9,0);

%\cercle{center2a}{1cm}{180}{-180}
%\cercle{center2b}{1cm}{180}{180}
%\cercle{center2c}{1cm}{180}{-180}

%\cercle{center3a}{1cm}{180}{180}
%\cercle{center3b}{1cm}{180}{-180}
%\cercle{center3c}{1cm}{180}{180}

%\draw[thick] (6,1.5) -- (6,-1.5);
%\draw[thick] (12,1.5) -- (12,-1.5);
%%%%%%%%%%%%%  END OLD  %%%%%%%%%%%%%%

\end{tikzpicture}
}
\caption{Left: This union of two circles is not of positive reach. Right: A set of positive reach.} \label{fig:reach}
\end{figure}
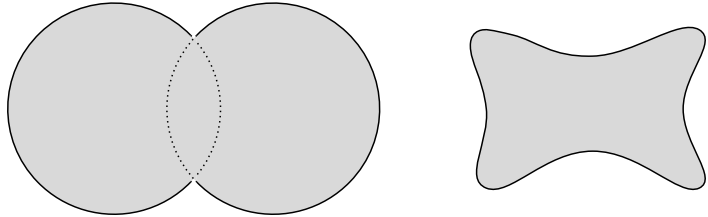

\subsection{Explicit stratification strategies for arbitrary $N$.}\label{subsect:2.4}
Next, we suggest and motivate a general and versatile construction of partitions for arbitrary $N$. 
We define partitions of the unit square generated by parallel lines which are orthogonal to the diagonal of the square. As a special case we consider the partitions $\bOmega_{\ast}^{(N)}$ which are equivolume; see Fig.~ \ref{fig:simplePartition} (right).
In Section \ref{sec4:num} we present numerical evidence that stratified samples based on such partitions improve the expected $\cL_2$-discrepancy of an $N$-point Monte Carlo sample roughly by a factor of two. 
As a comparison, we show in Example 1 in Section \ref{example1} that samples based on vertical strip partitions improve an $N$-point Monte Carlo sample by a factor of $5/3$.

Importantly, this construction enables us also to systematically study the role of the equivolume property. 
In a first step, in Example 3 in Section \ref{sec4:ex3} we improve the minimal convex equivolume partition for $N=2$ obtained in Example 2 in Section \ref{sec:equivolConvexN=2} by shifting the separating line along the diagonal. In Section \ref{sec4:ex4} we extend this analysis to the case $N=3$. We parametrise all such partitions into three sets and determine the minimal partition among them. It turns out that these partitions into three sets have a rich and interesting global structure with respect to their expected discrepancy. 

Finally, we have examples of partitions within this family and for small $N$ that show that it is possible to improve classical jittered sampling by relaxing the equivolume constraint. 
%\textcolor{red}{[add numerical results in more detail.]}

\subsection{Conclusions and open questions}\label{subsect:2.5}
In conclusion, our results show that if partitions are needed to generate stratified samples for arbitrary $N$,
we suggest to use lines that are orthogonal to the diagonal of the unit square. Within this family it seems that equivolume partitions  $\bOmega_{\ast}^{(N)}$ are a reasonably good pick; see Section \ref{sec4:num} for details.
%However, since the sets get very thin as $N$ increases, we believe that this construction can be improved with a clever rule that preserves a 
%\textcolor{blue}{[[Conclusions: (1) if partition is needed for arbitrary $N$, we suggest equivolume parallel lines. Seems to be a factor 2 better than random points, and is asymptotically ud. (2) as $N=3$ example shows, this may even be improved if we drop equi/volume.  (3) jittered is very good, but there may even be better partitions as example for $N=4$ shows if we drop equi-volume.]]}
Secondly, our examples for $N=2,3,4$ show that the expected discrepancy can be improved if we drop the equivolume property. This is in line with the results from \cite{stefan2} and deserves further attention. It certainly relates to the well-known general observation that the $\cL_2$-discrepancy \emph{exaggerates the importance of points lying close to the origin} (see \cite[pg 13f]{mat}).
\begin{question} \label{qu1}
%Apart from the equivolume assumption, which other properties of sequences of partitions can improve the expected discrepancy of point sets obtained from stratified sampling?
Are there properties of sequences of partitions, other than equivolume, that improve asymptotically the expected discrepancy of Monte Carlo sampling?
\end{question}

Our example for $N=4$ supports the idea brought forward in \cite{stefan1} that classical jittered sampling may not give the lowest expected discrepancy for large $N$.

\begin{question} \label{qu2}
Is there an infinite family of partitions that generates point sets with a smaller expected discrepancy than classical jittered sampling for large $N$?
\end{question}

%\textcolor{red}{[formulate conjecture/question about new problem of relating the order of convergence to the diameter of the sets in the partition.]}
%\textcolor{blue}{[Again I am not sure. I think we know too little to pose a question, which might be rather easy to answer if we dive into things. Some preliminary ideas of mine work with the \emph{mean} diameter, which would be a different story. I would rather omit it.]}
%\textcolor{red}{[an aspect might also be a uniform diameter for all sets vs a different diameter for different sets. so we could just ask about the relationship between diameter and expected discrepancy.]}

%An interesting aspect concerning Question \ref{qu3} is that we experimented with different types of partitions for $N=3$ and it turns out that the partition shown in Fig. \ref{N3min} has a smaller expected discrepancy than the best $Y$-shaped partitions we found. 

%%%%%%%%%%%%%%%%%%%
% ud sequences of partitions
%%%%%%%%%%%%%%%%%%%
\section{Proofs and examples} \label{sect:3}
\subsection{Proofs for Section \ref{sec2.1}} \label{sec3.1}
We now give proofs of the results in Subsection \ref{sec2.1} using the notations and notions introduced there. On several occasions we will need the following lemma, which essentially is a reformulation of the fact that a distribution (in the probabilistic sense) is uniquely determined by its cumulative distribution function, also in the multivariate case; see e.g. \cite[Example 1.44]{klenke}. Indeed, the proof of the following lemma is based on this fact, if the function involved is split into positive and negative part and the total integrals are normalized. 

\begin{lemma}\label{lem1A}
An integrable function  $f:[0,1]^d\to \R$ is almost everywhere determined if its integrals 
$\int_{[0,\bx]}f(\by)\dd\by$ are known for almost all $\bx\in [0,1]^d$. 

In other words, $\int_{[0,\bx]}f(\by)\dd\by=0$ for almost all  $\bx\in [0,1]^d$ implies $f(\bx)=0$ for almost all $\bx\in [0,1]^d$. 
\end{lemma}
We are now in a position to show the announced characterization of equivolume partitions in terms of the unbiasedness of the proportions $Z_\bx(\cP)$ in \eqref{eq:Zx}. 
 
\begin{proposition}\label{prop:1}
For  a partition $\bOmega$ of $[0,1]^d$ into $N$ Lebesgue sets $\Omega_1,\ldots, \Omega_N$ of positive volume
the following three statements are equivalent.
\begin{enumerate}
	\item[(i)] $\bOmega$ is equivolume.
	\item[(ii)] $\E Z_\bx(\cP)=\big|[0,\bx]\big|$ for all $\bx\in [0,1]^d$. 
	\item[(iii)] $\E Z_\bx(\cP)=\big|[0,\bx]\big|$ for almost all $\bx\in [0,1]^d$.
\end{enumerate}
\end{proposition}

\begin{proof} 
The bias is 
\begin{equation}
    \label{eq1}
\E Z_\bx{(\cP)}-|[0,\bx]|= \frac1N \sum_{i=1}^N \left[\frac{|\Omega_i\cap [0,\bx]|}{|\Omega_i|}-{N|\Omega_i\cap [0,\bx]|}\right]= \frac1N \sum_{i=1}^N u_i|\Omega_i\cap [0,\bx]|,
\end{equation}
where 
\[
u_i=\frac1{|\Omega_i|}-N,\qquad i=1,\ldots,N. 
\]
If (i) holds, the vector  $\bu=(u_1,\ldots,u_N)$ is the zero vector and the bias \eqref{eq1} vanishes for all $\bx\in [0,1]^d$. Hence, (i) implies (ii).

Clearly (ii) implies (iii), so it remains to assume (iii) and deduce (i). Assumption (iii) implies
 that \eqref{eq1} vanishes for almost all $\bx\in [0,1]^d$. This implies
$\int_{[0,\bx]} f(\by)d\by=0$ for almost all $\bx\in [0,1]^d$, where we have put 
$
 f=\sum_{i=1}^N u_i 1_{\Omega_i}.
$
Lemma \ref{lem1A} implies $f=0$ almost everywhere on $[0,1]^d$, and as $\bOmega$ is a partition, $\bu=0$. Hence the partition is equivolume. 
\end{proof}

\subsection{Proofs for Section \ref{sec2.2}} \label{sec3.2}
\begin{proof}[Proof of Theorem \ref{thm1}]
	Let $p>1$ be given. Using the variable $Z_\bx=Z_\bx(\cP)$ from \eqref{eq:Zx} and applying Tonelli's theorem we see that
	\[
	\E {\cL}_p(\cP_\bOmega)^p=\int_{[0,1]^d} \E\left(Z_\bx-|[0,\bx]|\right)^p \dd\bx,
	\]
	so the equivolume assumption and Proposition \ref{prop:1} yield
	\begin{equation}\label{eq:pthcnteredmean} 
	\E {\cL}_p(\cP_\bOmega)^p=\int_{[0,1]^d} \M_p(Z_\bx) \dd\bx. 
	\end{equation}
	Here, 
	\[
	\M_p(Y)=\E\big|Y-\E Y\big|^p
	\]
	is the $p$th centered moment of a random variable $Y$. The variable  $NZ_\bx$ is the sum of $N$ independent (but  not identically distributed) Bernoulli variables with success probabilities $q_1(\bx),\ldots,$ $q_N(\bx)$, where 
	\begin{equation}\label{eq:qi}
	q_i(\bx)=\frac{|\Omega_i\cap [0,\bx]|}{|\Omega_i|}=N|\Omega_i\cap [0,\bx]|.
	\end{equation}
	The distribution of $NZ_\bx$ is usually called \emph{Poisson-binomial distribution} with $N$ trials and parameter vector $\bq(\bx)=(q_1(\bx),\ldots,q_N(\bx))$. Its mean is $\sum_{i=1}^N q_i(\bx)=N|[0,\bx]|$. 
	
	Setting $U_\bx=(\#\{X_1,\ldots,X_N\}\cap [0,\bx])/N$ with uniform i.i.d.~random variables $X_1,\ldots,X_N$ 
	in $[0,1]^d$ and using similar arguments as above yields correspondingly  
	\begin{align}\label{eq:MC}
	\E {\cL_p}(\cP_N)^p=\int_{[0,1]^d} \M_p(U_\bx) d\bx.
	\end{align}
	The variable $NU_\bx$ has a binomial distribution with $N$ trials and success probability $|[0,\bx]|$. Its mean is therefore coinciding with the mean of $NZ_\bx$. 
	
	We now use the fact that among all Poisson-binomial distributions with given mean, the binomial is the largest one in convex order. This is formalized in \cite[Theorem 3]{hoeffding} (see also the paragraph directly after the statement of this theorem) and implies 
	\begin{equation}\label{eq:hoeffding} 
	\M_p(Z_\bx)\le \M_p(U_\bx)
	\end{equation} with equality if and only if $Z_\bx$ has a classical binomial distribution, that is, if and only if $q_1(\bx)=\cdots=q_N(\bx)=|[0,\bx]|$. Integrating \eqref{eq:hoeffding} with respect to $\bx$ now yields \eqref{eq:Florian}  if we can exclude the equality case.

	Equality in  \eqref{eq:Florian} would imply equality in \eqref{eq:hoeffding}, and thus 
	$N|\Omega_i\cap [0,\bx]|=|[0,\bx]|$, $i\in\{1,\ldots,N\}$,
	for almost all $\bx\in [0,1]^2$. Hence 
	$\int_{[0,\bx]} 1_{\Omega_i}(\by)d\by=\int_{[0,\bx]}\frac1N d\by$ for almost all $\bx\in [0,1]^d$ and all $i\in \{1,\ldots,N\}$.
	Lemma \ref{lem1A} implies $1_{\Omega_i}=1/N$,
	which is not possible as $N\ge 2$. 
\end{proof}

It is worth emphasizing the special case $p=2$ of \eqref{eq:pthcnteredmean}, which has been used more or less explicitly and generally in the existing literature, as it implies that the mean $\cL_2$-discrepancy can be described as sum of contributions from the individual sample points. Also for $p=4$ an explicit integral representation can be stated.
\begin{proposition}\label{prop:3}
	Let $\bOmega$ be an equivolume partition and let $q_i(\bx)$, $i=1,\ldots,N$ be defined by \eqref{eq:qi}.  
Then 
	\begin{equation*}
\E {\cL}_2(\cP_\bOmega)^2=\frac1{N^2}\sum_{i=1}^N \int_{[0,1]^d} Q_i(\bx) \dd\bx,
\end{equation*}
with $Q_i(\bx)=q_i(\bx)\big(1-q_i(\bx)\big)^2+q_i(\bx)^2\big(1-q_i(\bx)\big)=q_i(\bx)\big(1-q_i(\bx)\big)$, 
and 
\begin{equation*}
\E {\cL}_4(\cP_\bOmega)^4=
\frac1{N^4}\sum_{i=1}^N \int_{[0,1]^d} R_i(\bx)\dd\bx+\frac6{N^4}\sum_{i=1}^N\sum_{j=1 \atop j\ne i}^N
\int_{[0,1]^d} Q_i(\bx)Q_j(\bx)\dd\bx,
\end{equation*}
where $R_i(\bx)=q_i(\bx)\big(1-q_i(\bx)\big)^4+q_i(\bx)^4\big(1-q_i(\bx)\big)$. 
\end{proposition}	

\begin{proof}
According to  \eqref{eq:pthcnteredmean} with $p=2$, we have 
	\begin{equation}\label{eq:1}
\E {\cL}_2(\cP_\bOmega)^2=\int_{[0,1]^d} \Var\big(Z_\bx(\cP)\big) \dd\bx, 
\end{equation}
	where  $Z_\bx(\cP)$ is given in \eqref{eq:Zx}. We have already seen that 
	$N Z_\bx(\cP)$  is the sum of the independent Bernoulli variables $Y_{i}=1_{[0, {\bf x}[}(X_i)$ with success probabilities $q_1(\bx),\ldots,$ $q_N(\bx)$, respectively, so $N^2\Var\big(Z_\bx(\cP)\big) =\sum_{i=1}^N  q_i(\bx)\big(1-q_i(\bx)\big)$. This can be inserted into \eqref{eq:1} to obtain the first claim.

	For the second claim, let $W_i=Y_i-\E Y_i$, $i=1,\ldots,N$, and note that 
	\begin{align}\label{eq:Last}
	N^4\M_4\!\big(Z_\bx(\cP)\big)=\sum_{i_1,\ldots,i_4=0}^N  \E (W_{i_1}\cdots W_{i_4})=\sum_{i=1}^N \E W_i^4+{4\choose 2}\sum_{i=1}^N\sum_{j=1 \atop j\ne i}^N\E W_i^2\E W_j^2.
	\end{align}
	As  
	\[
	P(W_i=w)=\left\{\begin{array}{ll}
	q_i(\bx), &\text{if } w=1-q_i(\bx),\\
	1-q_i(\bx), &\text{if } w=-q_i(\bx).
	\end{array}\right.
	\]
	we have $ \E W_i^4=R_i(\bx)$ and $\E W_i^2=Q_i(\bx)$. 
%	, so 
%	\[
%\M_4\big(Z_\bx(\cP)\big)=\frac1{N^4}\sum_{i=1}^N R_i(\bx) +\frac6{N^4}\sum_{i=1}^N\sum_{j=1 \atop j\ne %i}^NQ_i(\bx)Q_j(\bx).
%\]	 
Inserting this into \eqref{eq:Last} and applying   \eqref{eq:pthcnteredmean} with $p=4$ yields the second assertion. 
\end{proof}

As an application of the previous proposition, we show that an equivolume partition minimizing the  mean $\cL_2$-discrepancy exists without any further regularity assumptions  when $d=1$ and $N=2$. 
	\begin{corollary}\label{CORnew}
		An equivolume partion $\bOmega=(\Omega_1,\Omega_2)$ of the unit interval $[0,1]$ minimizes the  mean $\cL_2$-discrepancy
		of its corresponding stratified point set among all equivolume partitions of two sets if and only if $\Omega_1$ coincides up to a set of measure zero with $[0,1/2]$ or %$\Omega_1$ coincides up to a set of measure zero 
		with $[1/2,1]$. 
	\end{corollary}

\begin{proof} 
	Let an equivolume partition $\bOmega=(\Omega_1,\Omega_2)$  of the unit interval be given. It is determined by the measurable set  $\Omega_1\subset [0,1]$ with $1$-dimensional Lebesgue measure $1/2$, as $\Omega_2=[0,1]\setminus \Omega_1$ (at least up to a set of measure zero). The functions $q_i(\cdot)$ in \eqref{eq:qi} are thus  $q_1(x)=2|\Omega_1\cap [0,x]|$ and $q_2=2x-q_1(x)$ and Proposition \ref{prop:3} implies 
		\begin{equation}\label{eq:l2}
			\E {\cL}_2(\cP_\bOmega)^2=\frac1{4}\int_{0}^1  2x-4x^2+2q_1(x)\big(2x-q_1(x)\big)\dd x=-\frac{1}{12}+
			\frac1{2}\int_{0}^1  g_x\big(q_1(x)\big)\dd x,
		\end{equation}
		where $g_x(q)=q(2x-q)$. Clearly $q\mapsto g_x(q)$ is strictly concave for all $x\in [0,1]$.

		It is easy to see that $\underline q(x)\le q_1(x)\le \overline q(x)$, where 
		\[
		\overline q(x)=2\big|[0,1/2]\cap [0,x]\big|\quad\text{  and }\quad
		\underline q(x)=2\big|[1/2,1]\cap [0,x]\big|
		\]
		for $x\in [0,1]$. Hence, there is an $\alpha_x\in [0,1]$ with 
		$q_1(x)=\alpha_x \underline q(x)+(1-\alpha_x)\overline q(x)$. Now, \eqref{eq:l2}, the  concavity of $g_x$ and the fact that $g_x\big(\underline q(x)\big)=g_x\big(\overline q(x)\big)=\max\{0,2x-1\}$ yield 
		\begin{align*}%\label{eq:l2}
			\E {\cL}_2(\cP_\bOmega)^2&\ge -\frac{1}{12}+
			\frac1{2}\int_{0}^1  \alpha_x g_x(\underline q(x))+(1-\alpha_x)g_x(\overline q(x))\dd x
			\\&=-\frac{1}{12}+
			\frac1{2}\int_{\frac12}^1 (2x-1)\dd x=\frac1{24}, 
		\end{align*}
		with equality if and only if $\alpha_x\in \{0,1\}$ holds for almost all $x\in (0,1)$ due to the strict concavity. As $q_1$ is continuous, this can only happen when $\alpha_x=1$ for all $x\in (0,1)$ or $\alpha_x=0$ for all $x\in (0,1)$. These two cases correspond to $q_1\in \{\underline q, \overline q\}$, and thus to the two stated choices of $\Omega_1$. 
	
\end{proof}

\subsection{Example 1: Illustration of partition principle} \label{example1}
For illustration, we derive the mean $\cL_2$-discrepancy of an $N$-point Monte Carlo sample in $[0,1]^d$. Using 	\eqref{eq:MC}, the i.i.d.~property of the sampling points $\bX_1,\ldots,\bX_N$ and 
\[
\Var(1_{[0,\bx[}(\bX_i))=\big|[0,\bx[\big|\big(1-\big|[0,\bx[\big|\big)
\]
we  obtain 
\begin{align*}
\E {\cL_2}(\cP_N)^2 = \int_{[0,1]^d} \Var\Big(\frac1N\sum_{i=1}^N 1_{[0,\bx[}(\bX_i) \Big)\dd\bx 
=\frac1N\int_{[0,1]^d} \big|[0,\bx[\big|\big(1-\big|[0,\bx[\big|\big)\dd\bx.
\end{align*}
The latter integral equals $\int_{[0,1]^d} \Big(\prod_{i=1}^d x_i-\prod_{i=1}^d x_i^2\Big)\dd\bx$ and can be evaluated explicitly. One obtains 
\begin{align}\label{eq:MCLd}
\E {\cL_2}(\cP_N)^2 =\Big[\frac1{2^d}-\frac1{3^d}\Big]\frac1N. 
\end{align}
In particular, for $d=2$, we get 
\begin{align}\label{eq:MCL2}
\E {\cL_2}(\cP_N)^2 =  \frac{5}{36N}. 
\end{align}

We now compare this with the mean $\cL_2$-discrepancy of a stratified sample $\cP_{\mathrm{vert}}$ based on the \emph{vertical strip partition} in Fig.~\ref{fig:simplePartition} generalized to arbitrary $d\ge 2$ by putting
\[
\Omega_i=\left\{\bx=(x_1,\ldots,x_d)\in [0,1]^d: \frac{i-1}N\le x_1\le \frac{i}{N}\right\}
\]
for $i=1,\ldots,N$. The partition $\bOmega=(\Omega_1,\ldots,\Omega_N)$ is clearly equivolume. 
  For given $\bx=(x_1,\ldots,x_d)\in ]0,1]^d$ we let $\bar{\iota}:=\lfloor N x_1 \rfloor$% be the index of the (first) partitioning set that contains $\bx$
, and obtain for the success probabilities introduced in the proof of Theorem \ref{thm1}
\[ 
q_i=q_i(\bx)= N{|\Omega_i \cap[0,\bx[ |} = 
\prod_{j=2}^N x_j\times 
\begin{cases}
1 & i \le\bar{a}, \\[3pt]
N x_1 - \bar{\iota}  &  i = \bar{\iota}+1,\\[3pt]
0 & i > \bar{\iota}+1.
\end{cases}
\]
Due to independence, the relative number of points $Z_{\bx}(\cP_{\mathrm{vert}})$ given by \eqref{eq:Zx} has variance
\begin{align*}
\M_2\big( Z_{\bx}(\cP_{\mathrm{vert}})\big) &=\Var \left( Z_{\bx}(\cP_{\mathrm{vert}})\right ) = \frac{1}{N^2} \sum_{i=1}^N q_i (1-q_i) \\
 &= \frac{1}{N^2} \left( q_{\bar{\iota}+1}(1-q_{\bar{\iota}+1}) + \bar{\iota}\prod_{j=2}^N x_j \big(1-\prod_{j=2}^N x_j\big) \right).
\end{align*}
Therefore, \eqref{eq:pthcnteredmean}  
yields 
\begin{align*}
\E {\cL_2}(\cP_{\mathrm{vert}})^2 &=\frac1{N^2}\int_{[0,1]^{d}}
\Big[
N\prod_{j=1}^N x_j-[(Nx_1-\lfloor N x_1 \rfloor)^2+\lfloor N x_1 \rfloor]\prod_{j=2}^N x_j^2\Big]\,\,\dd \bx\\
&= \frac1{2^dN}- \frac1{3^{d-1}N^2}\int_0^1\left[\big(Nx-\lfloor N x \rfloor\big)^2+\lfloor N x \rfloor\right]\,\dd x.
\end{align*}
The one-dimensional integral on the right hand side of this chain of equations evaluates to 
\[
\sum_{k=0}^{N-1} \int_{\frac kN}^{\frac{k+1}N}\big[(Nx-k)^2+k\big]\,\dd x=\frac{1}{3}+\frac{N-1}{2}=\frac{3N-1}{6}.
\]
Putting things together, we arrive at 
\begin{align}\label{eq:vertStrip}
\E {\cL_2}(\cP_{\mathrm{vert}})^2 =\Big[\frac1{2^d}-\frac{3N-1}{2N}\frac{1}{3^d}\Big]\frac1N<
\E {\cL_2}(\cP_N)^2,
\end{align}
where \eqref{eq:MCLd} and $N\ge 2$ was used. This confirms the general result that equivolume stratification is always strictly better than Monte Carlo sampling. It also shows that this stratification scheme has the same asymptotic order (namely $1/N$) as Monte Carlo sampling, but a uniformly better leading constant: for instance, when $d=2$ we get 
 \[
 \E {\cL_2}(\cP_{\mathrm{vert}})^2=\frac{3N+2}{36N^2}\approx \frac35 \E {\cL_2}(\cP_N)^2
 \]
 for large $N$.

%%%%%%%%%%%%%%%%%%%
%  Optimal partitions
%%%%%%%%%%%%%%%%%%%
\subsection{Proofs for Section \ref{subsect:2.3}} \label{subsect:3.3}
%In this section only the Lebesgue measure in $\R^d$ will be denoted by $\lambda_d$.
 Fix $A\subset \R^d$. We let $\inter A$ and $\bd A$ be the interior and the boundary of $A$, respectively. For $\eps>0$ the \emph{$\eps$-parallel set} 
\[
A_\eps=\{x\in \R^d: \inf_{y\in A}\|x-y\|\le  \eps\}
\]
consists of all points with a distance at most $\eps$ from $A$.  
We recall that a set $A\subset \R^d$ is said to have \emph{reach} $r>0$ if for any $0<\eps<r$ and $x\in A_\eps$ there is a point $y\in A$ such that $\|x-y\|<\|x-z\|$ for all $z\in A\setminus\{y\}$. 

The family of non-empty compact sets will be endowed with the Hausdorff metric $d_\mathrm{H}$ given by 
\[
d_\mathrm{H}(K,K')=\inf\{\eps>0: K\subset K_\eps',  K'\subset K_\eps\},
\]
where $\emptyset\ne K,K'\subset \R^d$ are compact. Let $\cC$ be the family of nonempty compact sets in $[0,1]^d$, and 
for $r>0$ let $\cR_r$ be the subfamily of sets with reach at least $r$. The latter contains $\cK$, the family of 
non-empty compact convex subsets of $[0,1]^d$. Crucial for our line of arguments is the fact that all three families are compact in the Hausdorff metric. This statement for $\cK$ is the famous \emph{Blaschke selection theorem} \cite[Theorem 1.8.7]{schneider}; a proof for $\cC$ and  $\cR_r$ can be found in \cite[Theorem 1.8.5]{schneider} and 
\cite[Remark 4.14]{Federer}, respectively. As a reference of convex geometric notions used in this section, we recommend \cite{schneider}.

Importantly, the volume functional is \emph{not} continuous on $\cC$ (for instance,  $[0,1]^d$ can be approximated by finite sets in the Hausdorff metric), but it is continuous on both $\cR_r$ and $\cK$. This can be seen by means of a Steiner-type result stating for $K\in \cR_r$ that
\begin{align}
\label{eq:Steiner}
|K_\eps\setminus K|=\sum_{k=0}^{d-1} \kappa_{d-k} V_{k}(K)\eps^{d-k}
\end{align}
holds for $0\le\eps<r$; see \cite{Federer}. Here, $\kappa_{d-k}$ is the volume of the Euclidean unit ball in $\R^{d-k}$, and $V_k(K)\in \R$ is the $k$th \emph{total curvature measure} (also called \emph{intrinsic volume} when applied to convex sets). We use repeatedly that $K\mapsto V_k(K)$ is continuous on $\cR_r$. 
These and more results on sets of positive reach can be found in 
\cite{Federer}; see also the survey \cite{Thaele}, where an outline of the history, newer results and 
additional references on the matter can be found. 

\begin{proposition}\label{prop:compact} 
	Let $N\ge1$ and $r>0$ be fixed. 	
	The family $\mathfrak{P}_N(r)$ of all equivolume  partitions of $[0,1]^d$ consisting of $N$ sets in $\cR_r$ is compact. 
	
	The same holds true for the family $\mathfrak{P}_N^{\mathrm{conv}}$ of all equivolume partitions of $[0,1]^d$ consisting of $N$ convex sets. 
\end{proposition}
\begin{proof}
	 Clearly, the family of equivolume partitions of sets in $\cR_r$  is a subset of the Cartesian product $\cR_r^N$, 
	%which will be endowed with the metric 
	%\[
	%d_{\mathrm H,N}\big((K_1\ldots,K_N),(K_1'\ldots,K_N')\big)=\max_{i=1}^N d_{\mathrm H}(K_i,K_i'),
	%\] 
%$(K_1\ldots,K_N),(K_1'\ldots,K_N')\in \cK^N$, 
%which generates the product topology. \medskip 
%	The family of all partitions in the Lemma is  
more precisely, 
\begin{align}\label{eq:partDef}
\mathfrak{P}_N(r)=\big\{(K_1,\ldots,K_N)\in \cR_r^N: \bigcup_{i=1}^K K_i=[0,1]^d, \text{ and } 
|K_1|=\cdots=|K_N|=\frac{1}{N}  \big\}.
\end{align}
In fact, assume that $(K_1,\ldots,K_N)$ is an element of the right hand side of \eqref{eq:partDef}. If there was a set $K_j$ overlapping with $\bigcup_{i\ne j}^N K_i$, we would have 
\[
  0<\big|K_j\cap \bigcup_{i\ne j}^N K_i\big|= 
   |K_j|+ \big|\bigcup_{i\ne j}^N K_i\big|-\big|K_j\cup \bigcup_{i\ne j}^N K_i\big|\le \frac1N+(N-1)\frac1N-1=0,
\]
a contradiction.

From the definitions it is clear that $(K,M)\mapsto K\cup M$ is Lipschitz continuous if the product space is endowed with the maximum norm of the marginal metrics. 
Hence, the right hand set in \eqref{eq:partDef} is closed in $\cR_r^N$, as it is defined by means of continuous functions. As $\cR_r^N$ is compact, $\mathfrak{P}_N(r)$ is compact too. 

Finally, $\cK^N$ is compact implying the compactness of $\mathfrak{P}_N^{\mathrm{conv}}=\mathfrak{P}_N(r)\cap \cK^N$. 
\end{proof}

%It should be noted that the convex partitions are polygonal inside $K$: if $(K_1,\ldots,K_N)\in\cP_N$ and $i\in\{1,\ldots,N\}$, then $\bd K_i\cap\inter K$ coincides with the intersection of a convex polytope  with $K$; see  \cite[Proof of Lemma 10.1.1]{SW} for a similar argument. In particular, when $K$ is a convex polytope, all $K_i$ are also convex polytopes. 
%

We now show that an average discrepancy of a stratified sample is continuous as a function of the partitioning sets. 
In the following, $\Delta$ stands for a measure of discrepancy, and can be the $\cL_p$-discrepancy or any other, which satisfies the rather weak assumptions below. 

\begin{proposition}\label{lem1}
	
	Fix $r>0$. If  $\Delta:([0,1]^d)^N\to \R$ is measurable and bounded, 
	then $\phi_\Delta:\cR_r^N\to \R$ with 
	\[
	\phi_\Delta(K_1,\ldots,K_N)=\int_{([0,1]^d)^N}\Delta(\bx_1,\ldots,\bx_N)
        \prod_{i=1}^N %\tfrac{
        \1_{K_i}(\bx_i)
    \,\dd(\bx_1\ldots,\bx_N)
	\]
	is Lipschitz continuous. 
\end{proposition}
\begin{proof}
	We let $\|\cdot\|_\infty$ be the $L^\infty$-norm of bounded functions  on $([0,1]^d)^N$. We will use the bound 
	\begin{equation}
	\label{vAM}
	\left|\prod_{i=1}^N t_i-\prod_{i=1}^N t_i'\right|\le \sum_{j=1}^N |t_j-t_j'|,
	\end{equation}
	which holds for all $t_1,\ldots,t_N,t_1',\ldots,t_N'\in \{0,1\}$.  Let $0<\eps<1$. 
	  If  $d_H(K_i,K_i')\le \eps$ for all $i=1,\ldots,N$, the indicators $\1_{K_i}$ and $\1_{K_i'}$ coincide on 
	  the complement of 
	\[M_i=
	\left[(K_i)_\eps\setminus {K_i}\right]\cup
	\left[(K_i')_\eps\setminus {K_i'}\right].
	\]
	By Steiner's formula \eqref{eq:Steiner}, we have 
	\begin{equation*}
	|M_i|\le \sum_{k=0}^{d-1}\eps^{d-k}\kappa_{d-k} [V_k(K_i)+V_k(K_i')]. 
	\end{equation*}
	As $V_k$ is continuous and $\cR_r$ is compact, $V_k(K_i)+V_k(K_i')\le 2\max_{M\in \cR_r} V_k(M)<\infty$, so 
	\begin{equation}
\label{epsbound}
	|M_i|\le c \eps
	\end{equation}
	for all $i=1,\ldots,N$, 
	with a constant $c$ that does not depend on $(K_1,\ldots,K_N,K_1',\ldots,K_N')$.  
	Thus, by \eqref{vAM} and \eqref{epsbound}, we have 
	\begin{align*}
	|\phi_\Delta(K_1,\ldots,K_N)-\phi_\Delta(K_1',\ldots,K_N')|%&\le \|\Delta\|_\infty\sum_{i=1}^{N} \lambda_n(M_i)
	%\\&
	\le cN \|\Delta\|_\infty \eps. 
	\end{align*}
	This implies the claimed continuity.
\end{proof}

\begin{proposition}\label{prop:4}
	Let $\Delta$ be as in Proposition \ref{lem1} and fix $r>0$. For any finite equivolume partition $\bOmega=(\Omega_1,\ldots,\Omega_N)$ of $[0,1]^d$ with sets in $\cR_r$, 
	let $\cP_\bOmega=\{\bX_1,\ldots,\bX_N\}$ be the corresponding stratified sample. 
	
	Then $\E \Delta(\bX_1,\ldots,\bX_N)$ is continuous as a function of $\bOmega\in\mathfrak{P}_N(r)$. 
\end{proposition}
\begin{proof}
	As the partitions are equivolume, we have $|\Omega_i|=1/N$ for all $i$, so 
 \[
 \E \Delta(\bX_1,\ldots,\bX_N)
 =N^N\phi_\Delta(\Omega_1,\ldots,\Omega_N),
\]	
and the claim follows from Proposition \ref{lem1}. 
\end{proof}

\begin{proof}[Proof of Theorem \ref{thm:existenceReach}.]
	Assume that $\Delta:([0,1]^d)^N\to \R$ is measurable and bounded. Proposition \ref{prop:4} thus implies that $\E \Delta(\cP_{\bOmega})$ is a continuous function of $\bOmega\in\mathfrak{P}_N(r)$, where   $\cP_{\bOmega}$ is the corresponding stratified sample.
	
	As $\mathfrak{P}_N(r)$  and its subset $\mathfrak{P}_N^{\mathrm{conv}}$  are both compact by Proposition \ref{prop:compact},  $\E\Delta(\cP_\bOmega)$ attains minima on either set. This shows the assertion. 
\end{proof}

Corollary \ref{thm:existenceConvex} now follows directly from Theorem \ref{thm:existenceReach}. In fact, for $1\le p<\infty$ the function 	\[
	\Delta_p(\bx_1,\ldots,\bx_N)=\cL_p(\{\bx_1,\ldots,\bx_N\})^p
	\] is bounded.  
	It is also continuous due to a dominated convergence argument. Even simpler, the measurability and boundedness of $\Delta(\bx_1,\ldots,\bx_N)=D^*(\{\bx_1,\ldots,\bx_N\})$ follows directly from the definition. 
\medskip

It should be remarked that the above arguments do not rely on the choice of the unit cube as reference set. 
The results and proofs continue to hold with minor modifications if the  set $[0,1]^d$ is replaced by any fixed compact convex set $K\subset \R^d$.  

%\textcolor{red}{[[discuss role of equivolume again. in this draft we have found some answers to the remarks in old draft file!]]}

%\begin{tabular}{c}
%	\phantom{xxxxxxxxxxxxxxxxxxxxxxxxxxxxxxxxxxxxxxxxxxxxxxxxxxxxxxxxxxxxxxxxxxxxxxxxxx}\\
%	\hline
%\end{tabular}

\subsection{Example 2: Convex equivolume partitions into two sets}\label{sec:equivolConvexN=2}
Recall that $\mathfrak{P}_N^{\mathrm{conv}}$ is the family of all convex equivolume partitions of $[0,1]^d$ with $N$ elements. According to Theorem \ref{thm:existenceConvex}, there exists a partition in $\mathfrak{P}_N^{\mathrm{conv}}$  that minimizes the mean $\cL_2$-discrepancy. The following result determines this partition for $N=2$ and $d=2$. In this simple case, $\mathfrak{P}_N^{\mathrm{conv}}$  can be described with the help of a one-parameter model, which is relatively easy to analyze. The optimal partition $\bOmega_{*}^{(2)}$ is obtained by cutting $[0,1]^2$ into two congruent triangles by the anti-diagonal; see Fig.~\ref{fig:two}, right. 

\begin{center}
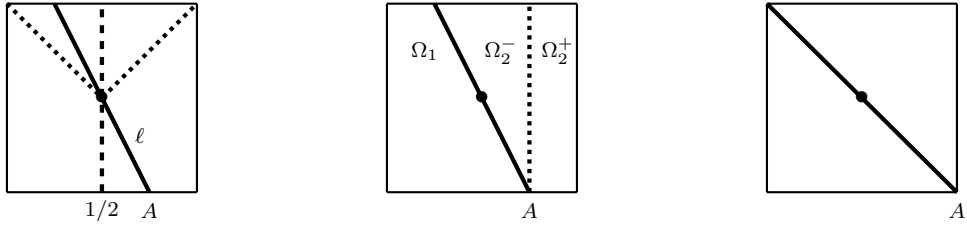
\begin{figure}[h!]
\begin{tikzpicture}[scale=2.5]
\draw [thick] (0,0) -- (1,0);
\draw [thick] (1,1) -- (1,0);
\draw [thick] (0,1) -- (1,1);
\draw [thick] (0,0) -- (0,1);
\draw [ultra thick] (0.25,1) -- (0.75,0);
\draw [ultra thick, dashed] (0.5,0) -- (0.5,1);
\draw [ultra thick, dotted] (0,1) -- (0.5,0.5) -- (1,1);
\node at (0.5, 0.5) {$\bullet$};
%\draw [thick] (0.5,0.5) circle (0.2cm);
%\draw [thick,->] (0.7,0.05001)(0.7,0.5);
%\draw [thick,->] (0.05001,0.5)(0.5,0.7);
%\draw (0.5,0.85) arc (90:118:0.35);
\node at (0.5, -0.1) {\scriptsize $1/2$};
\node at (0.75, -0.1) {\scriptsize$A$};
\node at (0.7, 0.3) {\scriptsize$\ell$};
%\node at (0.44,0.75) {\scriptsize $\varphi$};

%%%%%%%%%%%%%%%%

\draw [thick] (2,0) -- (3,0);
\draw [thick] (3,1) -- (3,0);
\draw [thick] (2,1) -- (3,1);
\draw [thick] (2,0) -- (2,1);
\draw [ultra thick] (2.25,1) -- (2.75,0);
\draw[ultra thick, dotted] (2.75,0) -- (2.75,1);
\node at (2.5, 0.5) {$\bullet$};

\node at (2.2, 0.75) {\scriptsize $\Omega_1$};
\node at (2.6, 0.75) {\scriptsize $\Omega_2^-$};
\node at (2.9, 0.75) {\scriptsize $\Omega_2^+$};
\node at (2.75, -0.1) {\scriptsize$A$};

%%%%%%%%%%%%%%%%%%%%%

\draw [thick] (4,0) -- (5,0);
\draw [thick] (5,1) -- (5,0);
\draw [thick] (4,1) -- (5,1);
\draw [thick] (4,0) -- (4,1);
\draw [ultra thick] (4,1) -- (5,0);
\node at (4.5, 0.5) {$\bullet$};
\node at (5, -0.1) {\scriptsize$A$};

\end{tikzpicture}
\caption{Left: Model for all convex partitions into two sets with equal volume. Middle: The three different regions considered for the case $A \in (1/2,1]$. Right: The partition $\bOmega_{*}^{(2)}$ of this family with the smallest expected discrepancy.} \label{fig:two}
\end{figure}
\end{center}

\begin{lemma} \label{lem:n2}
For $d=2$ we have 
$$ \min_{\bOmega\in \mathfrak{P}_2^{\mathrm{conv}}}  \E \cL_2^2(\cP_{\bOmega}) = 
	\E \cL_2^2(\cP_{\bOmega_{*}^{(2)}} ) = 0.05.$$
\end{lemma}

\begin{proof}
Let $\Omega_1$ and $\Omega_2$ be two convex sets  that partition $[0,1]^2$ and have the same content. 
By convexity, the intersection of  $\Omega_1$ and $\Omega_2$ is contained in a line $\ell$.  
The midpoint $p=(1/2,1/2)$ of $[0,1]^2$  must be contained in one of these sets, so let us assume $p\in \Omega_1$. 
As the reflection $\ell'$ of $\ell$ at $p$ is parallel to $\ell$, the reflection $\Omega_1'$ of $\Omega_1$ at $p$ must contain $\Omega_2$. By the equivolume property we have $|\Omega_2|=|\Omega_1|=|\Omega_1'|$, so $\Omega_2=\Omega_1'$ up to a Lebesgue-null set. As $\Omega_2$ and $\Omega_1'$ are closed and convex we must even have $\Omega_2=\Omega_1'$, so 
$\Omega_2$ is the reflection of $\Omega_1$ at $p$ and we conclude $p\in \ell$. 

As  $\bOmega_{*}^{(2)}$ and $\E \cL_2^2(\cP_{\bOmega})$ are unaltered if all sets in the partition are reflected at the main diagonal  of $[0,1]^2$, we may assume from now on that $\ell$ hits the $x$-axis in a point $A\in [0,1]$; see Fig.~\ref{fig:two} (left). Note that $\bOmega_{*}^{(2)}$ corresponds to $A=1$. For fixed $A$, we assume from now on that  $\Omega_1$ is the partitioning set that contains the left, vertical edge of the unit square.

To calculate the expected $\cL_2$-discrepancy, we use the formula 
\begin{align}\nonumber 
\E \cL_2^2 (\cP_\bOmega ) &= \frac{1}{72} + 2 \int_{[0,1]^2}f(\bx) x_1x_2 - f(\bx)^2 \ \dd \bx
\\&= \frac{1}{72} + 2 \int_{\Omega_2}f(\bx) x_1x_2 - f(\bx)^2 \ \dd \bx \label{eq:conv2equi}
\end{align}
from \cite{stefan2}, in which $f(\bx)= | \Omega_1 \cap [0,\bx] | $, and where we used for the second equality that the integrand vanishes whenever $\bx\in \Omega_1$.

We distinguish the three cases $A=[0,1/2)$, $A=(1/2,1]$ and $A=1/2$.
The special case $A=1/2$ gives the vertical strip partition for $N=2$ with expected discrepancy $1/18=0.055\ldots$ according to \eqref{eq:vertStrip}.

Now assume that $A=(1/2,1]$. In this case, the separating line  $\ell$ has the equation
$$ y = \frac{1}{1-2A} x - \frac{A}{1-2A}. $$
We partition $\Omega_2$ into the two sets $\Omega_2^+=\{(x,y)\in [0,1]^2: x\ge A\}$ and $\Omega_2^-=\Omega_2\setminus\Omega_2^+$; see Fig.~\ref{fig:two} (middle). For $\bx\in \Omega_2^+$  we have
$$ f(x,y) %= A y - \frac{1}{2}y\big(A - ((1-2A)y+A)\big)
=Ay-(2A-1)\frac{y^2}{2}, $$
so
\begin{align*}
 2 \int_{\Omega_2^+}f(\bx) x_1x_2 - f(\bx)^2 \ \dd \bx 
 =\frac1{120} (-2A^3-11A^2+10A+3),
\end{align*}
whereas for $\bx\in \Omega_2^-$ we have
\begin{align*}
f(x,y) 
%= x y - \frac{1}{2}(x- ((1-2A)y+A) ) \left(y - \left(\frac{1}{1-2A}\right)x - \frac{A}{1-2A} \right)
=(1-2A)\frac{y^2}{2}+Ay+\frac{(A-x)^2}{2(1-2A)},
\end{align*}
which results in
\begin{align*}
  2 \int_{\Omega_2^-}f(\bx) x_1x_2 - f(\bx)^2 \ \dd \bx 
	=\frac1{360} (1-2A)^2  (11+2A).
\end{align*}
In total, we obtain for $A \in (1/2,1]$
by inserting the contributions of both cases into \eqref{eq:conv2equi} that 
$$
\E \cL_2^2(\cP_{\bOmega} ) = \frac{1}{360}(2A^3 +3A^2-12 A+25),
$$
which attains  its minimum for $A=1$ with a value of $1/20=0.05$; see Fig. \ref{plot:fam1}.
Note that as $A\to 1/2$ this function approaches $0.055\ldots$.

We can analyse the last case $A=[0,1/2)$ in a similar fashion and obtain 
$$ 
\E \cL_2^2(\cP_{\bOmega} ) = \frac{1}{360}(-6A^3+3A^2-6A+23),
$$
which approaches its infimum as $A\rightarrow 1/2$ with a value of $0.055\ldots$; see Fig. \ref{plot:fam1}. This proves the assertion.
\end{proof}

\begin{figure}
\includegraphics[scale=0.6]{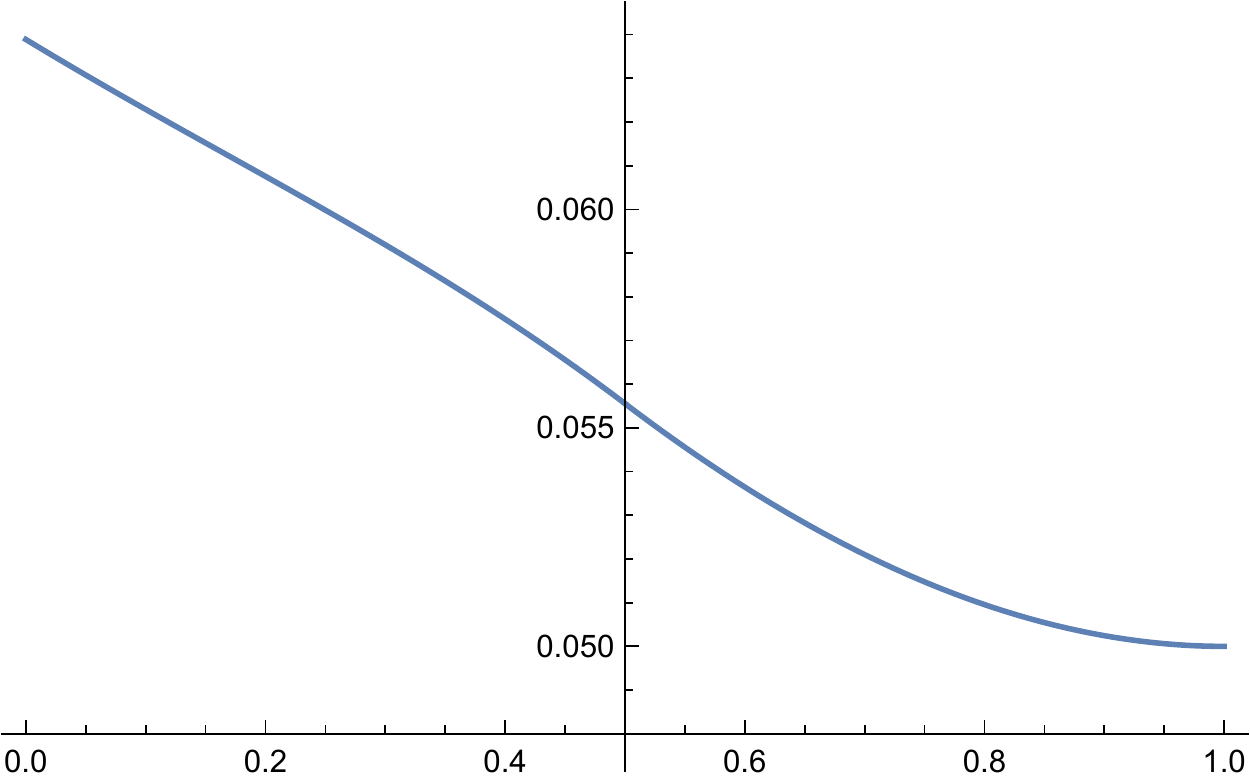} \ \ \
\includegraphics[scale=0.6]{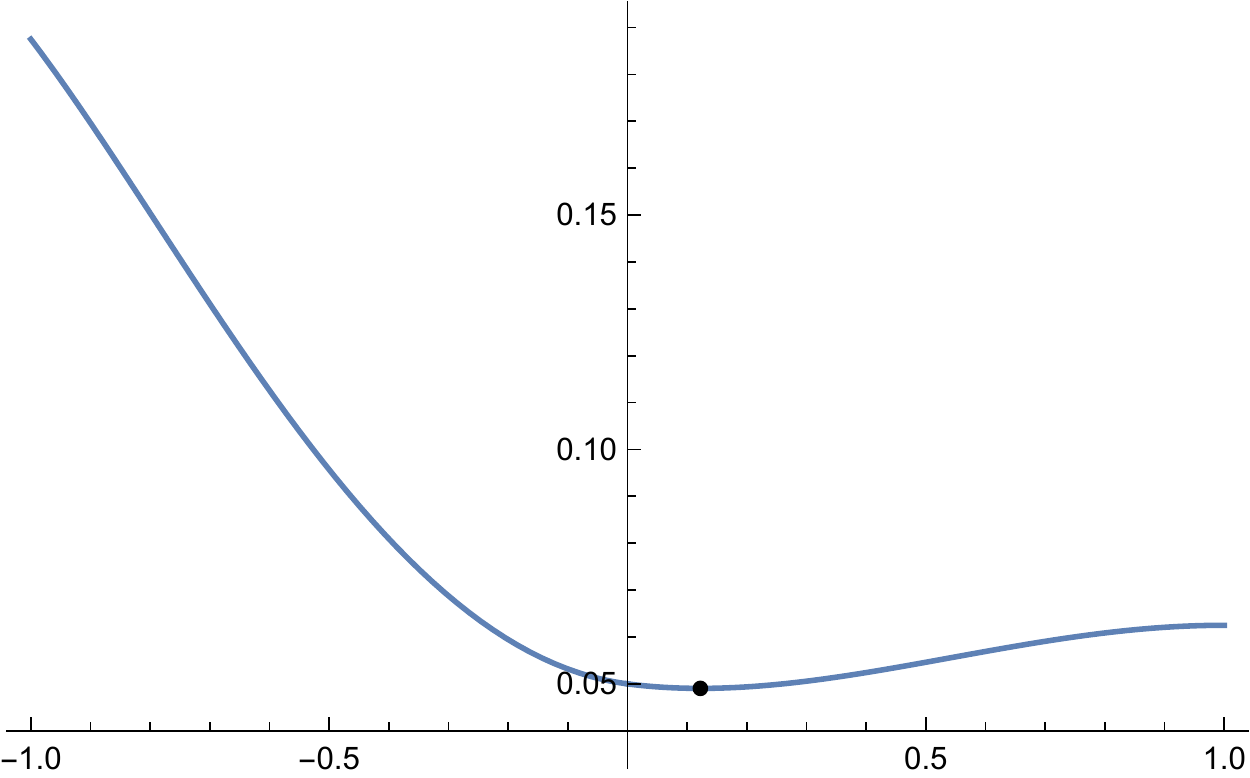}
\caption{Left: Expected discrepancy of partitions in Example 2. Parameter $A$ plotted on $x$-axis. Right: Expected discrepancy of partitions in Example 3. The $x$-axis encodes parameters $B=x+1$ in $[-1,0]$ and $A=x$ in $[0,1]$.}
\label{plot:fam1}
\end{figure}

%%%%%%%%%%%%%%%%%%%
%  Infinite family
%%%%%%%%%%%%%%%%%%%

\section{An infinite family and numerical results} \label{sect:4}

%%%%%%%%%%%%%%%%%%%
\subsection{An infinite family and special cases}\label{subsec:omega}
Motivated by the result of the previous section we define an $(N-1)$-parameter family of partitions of the unit square generated by parallel lines which are orthogonal to the diagonal of the square. For $N\in \N$ and a vector ${\bv v}=(v_1,\ldots,v_{N-1})\in [0,\sqrt{2}]^{N-1}$ with $0<v_1<v_2<\cdots<v_{N-1}<\sqrt{2}$ we define a partition $\bOmega_{\mathbf{v}}^{(N)}$ as follows. If $\ell_i$ denotes the line with slope $-1$ hitting the first closed quadrant and with distance $v_i$ from the origin, then $[0,1]^2\setminus \{\ell_{1},\ldots, \ell_{{N-1}}\}$ has $N$ connected components. Its closures are denoted by $\Omega_1,\ldots\Omega_N$, where $\Omega_i$ is positioned between $\ell_{{i-1}}$ and $\ell_i$ if we use the convention $v_0=0$ and $v_N=\sqrt 2$.  Examples are illustrated in Fig. \ref{fig:algorithm}.
We will often use the abbreviation $\bOmega_{v_1,\ldots,v_{N-1}}^{(N)}$ for the partition $\bOmega_{(v_1,\ldots,v_{N-1})}^{(N)}$.

An interesting special case are the equivolume partitions denoted by $\bOmega_{\ast}^{(N)}$. They are defined via
$\bOmega_{\ast}^{(N)}= \bOmega_{v_1,\ldots,v_{N-1}}^{(N)}$ with
$$ v_i =  \sqrt{\frac{i}{N}},$$
for $1 \leq i \leq \lfloor N/2 \rfloor$ and 
$$ v_i = \sqrt{2} - \sqrt{\frac{N-i}{N}},$$
for $\lfloor N/2 \rfloor +1 \leq i \leq N-1$.

As a side remark, we mention also the partition generated by a set of equidistant points; i.e. $v_i = \sqrt{2} i/N$ for a given $N>1$. This is a simple example of a family of partitions that is not equivolume for any $N$. However, by pairing complementary sets such that the union has volume $2/N$, it is possible to obtain equivolume partitions for every even $N$ with $N/2$ points into non-connected sets.

%%%%%%%%%%%%%%%%%%%%%%%%%%%%%%%
\subsection{Example 3: Relaxing the volume constraint I} \label{sec4:ex3}
We have seen in Example 2 that $\bOmega_{\ast}^{(2)}$ gives the lowest expected discrepancy among all convex equivolume partitions into two sets. We will now show that this partition can be improved if the equivolume condition is dropped by shifting the line along the diagonal, that is, by considering partitions in the class $\bOmega_{v}^{(2)}$ with $v\in [0,\sqrt2 ]$; see  Fig. \ref{N2min}.

\begin{lemma}\label{lem:diagVersch}
We have
$$ \min_{v \in [0,\sqrt{2}]}  \E \cL_2^2(\cP_{\bOmega_{v}^{(2)}} ) = \E \cL_2^2(\cP_{\bOmega_{v^{\ast}}^{(2)}} ) \approx  0.049,$$
for $v^{\ast}= 0.793398\ldots$. 
\end{lemma}
A comparison with Lemma \ref{lem:n2} shows that $\bOmega_{v^{\ast}}^{(2)}$ has a smaller mean $\cL_2$-discrepancy than any convex equivolume partion when $N=2$.
\begin{proof}
For a given point $v\in [0,\sqrt{2}]$ we denote the corresponding intersection of the line with the boundary of the square with $(A,1)$ if $v \in [\sqrt{2}/2,\sqrt{2}]$ and with $(B,0)$ for $v \in [0,\sqrt{2}/2]$.
If $v \in [\sqrt{2}/2,\sqrt{2}]$, then $A=\sqrt{2}v-1$. The two partitioning sets have volumes $1-\frac{(1-A)^2}{2}$ and $\frac{(1-A)^2}{2}$ and points on the separating line satisfy $y=-x+1+A$. In total, we obtain for $A \in [0,1]$ and in a similar fashion as in the proof of Lemma \ref{lem:n2} that
$$ 
\E \cL_2^2 (\cP_A) = \frac{-18 - 30 A + A^2 - 36 A^3 + 52 A^4 - 12 A^5 - 2 A^6}{-360 - 720 A + 360 A^2}.
$$
This function attains its minimum for $A=0.122034$ with value $0.04904\ldots$; see Fig. \ref{plot:fam1}. 

Next, if $v \in [0,\sqrt{2}/2]$, then $B=\sqrt{2}v$. The two partitioning sets have volumes $B^2/2$ and $1-B^2/2$ and points on the separating line satisfy $y=-x+B$. Similar to the previous considerations we split the integral into subintegrals and distinguish the different cases on which the intersections can be described with the same function. We obtain
\[\E \cL_2^2 (\cP_B) =% \frac{-135+ B(120 + B(175 - 2B(144-56B + B^3)))}{360 (-2+B^2)}=
 \frac{-135 + 120 B + 175 B^2 - 288 B^3 + 112 B^4 - 2 B^6}{360 (B^2-2)}
\]
for $B \in [0,1]$, which attains its minimum for $B=1$ with value $0.05$; i.e. the minimum is attained for the anti-diagonal; see Fig. \ref{plot:fam1}.
This concludes the proof.
\end{proof}

\begin{center}
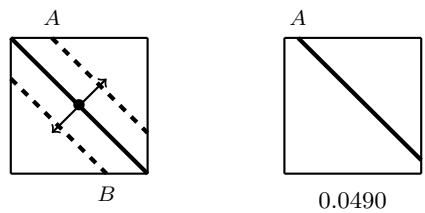
\begin{figure}[h!]
\begin{tikzpicture}[scale=1.8]
\draw [thick] (0,0) -- (1,0);
\draw [thick] (1,1) -- (1,0);
\draw [thick] (0,1) -- (1,1);
\draw [thick] (0,0) -- (0,1);
\draw [ultra thick] (0,1) -- (1,0);
\draw [ultra thick, dashed] (0.3,1) -- (1,0.3);
\draw [thick,->] (0.5,0.5)--(0.7,0.7);
\draw [thick,->] (0.5,0.5)--(0.3,0.3);
\draw [ultra thick, dashed] (0,0.7) -- (0.7,0);
\node at (0.5, 0.5) {$\bullet$};
\node at (0.3, 1.15) {\scriptsize $A$};
\node at (0.7, -0.15) {\scriptsize $B$};
%\draw [thick] (0.5,0.5) circle (0.2cm);
%\draw [thick,->] (0.7,0.05001)(0.7,0.5);

\draw [thick] (2,0) -- (3,0);
\draw [thick] (3,1) -- (3,0);
\draw [thick] (2,1) -- (3,1);
\draw [thick] (2,0) -- (2,1);
\draw [ultra thick] (2.1,1) -- (3,0.1);
\node at (2.5, -0.2) {\footnotesize 0.0490};
\node at (2.1, 1.15) {\scriptsize $A$};
\end{tikzpicture}
\caption{{\footnotesize Left: One-parameter model of partitions used in Lemma \ref{lem:diagVersch}. Right: The partition of this family with the smallest expected discrepancy.}} \label{N2min}
\end{figure}
\end{center}

%%%%%%%%%%%%%%%%%%%
\subsection{Example 4: Relaxing the volume constraint II} \label{sec4:ex4}
In this section we extend the results of Examples 2 and 3 to the case $N=3$. The partitions can still be analysed explicitly and it turns out that there is a unique partition that minimises the expected discrepancy; see Fig. \ref{N3min}.
However, the full analysis consists of a tedious case-by-case study and, hence, we do not fully outline the proof of this assertion in the following, but report the most interesting cases only.  We leave the analysis of the remaining cases -- which we carried out and which follows along the same lines as our proof -- to the interested reader.

In general, let $0< v_1 < v_2 < \sqrt{2}$ and denote the three sets of the partition with $\Omega_1, \Omega_2$ and $\Omega_3$, as described in Subsection \ref{subsec:omega}. Associated to the three sets, there are three indicator functions $\chi_1, \chi_2$ and $\chi_3$ with
$$ \chi_j(x,y) := 
\begin{cases}
1 & \text{ if } \bv{X}_j \in [0,x]\times [0,y]
, \\[3pt]
0 & \text { else},
\end{cases} $$
where $\bv{X}_j$ is the random point in $\Omega_j$. 
%[As far as I see, $\#(x,y)$ in the next formulae is undefined.]}
Setting $$\#(x,y) := \chi_1(x,y) + \chi_2(x,y) + \chi_3(x,y),$$
we get for the expected $\cL_2$-discrepancy
$$ \E\left ( \frac{\#(x,y)}{3} - x y \right )^2 = \frac{1}{9} \E(\#(x,y) )^2 - \frac{2}{3} x y \E(\#(x,y)) + x^2 y^2, $$
and since $\#(x,y)$ is a Poisson-binomial distributed random variable we have that
$$\E(\#(x,y)) = q_1(x,y) + q_2(x,y) + q_3(x,y),$$
in which, as in \eqref{eq:qi},
$$ 
q_j(x,y) = \PP(\chi_j=1) =  \frac{| [0,x]\times [0,y] \cap \Omega_j |}{|\Omega_j|}.
$$
The analysis proceeds now via two levels of case distinctions. 
The first level concerns the actual partitions into three sets that we are considering; see Fig. \ref{fig:cases1}.
Within each of the four cases, we partition the unit square in dependence of $v_1$ and $v_2$ into sets in which the probabilities $q_j$ have the same closed form in $x,y$, thus providing closed formulas for $\E(\#(x,y))$ and $\E(\#(x,y))^2$; see Fig. \ref{fig:cases2}.
Next, we compute the expected discrepancy, which is an integral over the unit square, as a sum of integrals over the different closed form expressions according to the subcase-partition.

\begin{center}
\begin{figure}[h!]
\begin{tikzpicture}[scale=1.8]
%left
\draw [thick] (0,0) -- (1,0)--(1,1)--(0,1)--(0,0);
\draw [dashed,thick] (1,0) -- (0,1);
\draw [ultra thick] (0.3,0) -- (0,0.3);
\draw [ultra thick] (0.7,0) -- (0,0.7);
\node at (0.15,0.15) {$\bullet$};
\node at (0.35,0.35) {$\bullet$};
%middle1 

\draw [thick] (2,0) -- (3,0)--(3,1)--(2,1)--(2,0);
\draw [dashed, thick] (3,0) -- (2,1);
\draw [ultra thick] (2.2,1) -- (3,0.2);
\draw [ultra thick] (2,0.7) -- (2.7,0);
\draw [dotted] (2.7,0) -- (2.7,0.7) -- (2,0.7);
\node at (2.35,0.35) {$\bullet$};
\node at (2.6,0.6) {$\bullet$};

\node at (2.7,-0.15) {\footnotesize {A}};
\node at (3.15,0.2) {\footnotesize{B}};

%middle 2
\draw [thick] (4,0) -- (5,0)--(5,1)--(4,1)--(4,0);
\draw[dashed, thick] (5,0)--(4,1);
\draw [ultra thick] (4.6,1) -- (5,0.6);
\draw [ultra thick] (4,0.7) -- (4.7,0);
\draw [dotted] (4.7,0) -- (4.7,0.7) -- (4,0.7);
\node at (4.35,0.35) {$\bullet$};
\node at (4.8,0.8) {$\bullet$};

\node at (4.7,-0.15) {\footnotesize {A}};
\node at (5.15,0.6) {\footnotesize{B}};

%right
\draw [thick] (6,0) -- (7,0)--(7,1)--(6,1)--(6,0);
\draw[dashed, thick] (7,0)--(6,1);
\draw [ultra thick] (6.3,1) -- (7,0.3);
\draw [ultra thick] (6.7,1) -- (7,0.7);
\node at (6.65,0.65) {$\bullet$};
\node at (6.85,0.85) {$\bullet$};
\end{tikzpicture}
\caption{{\footnotesize The four cases we distinguish. The black dots correspond to the points $v_1$ and $v_2$.}} \label{fig:cases1}
\end{figure}
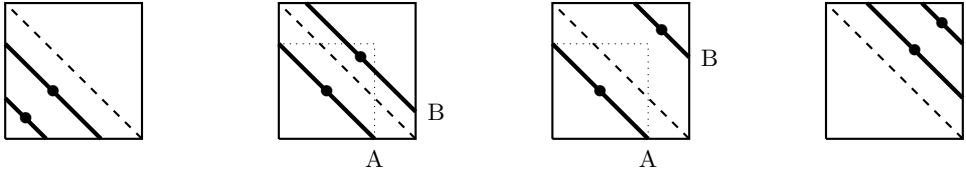
\end{center}

In the following, we first focus on the third case in Fig. \ref{fig:cases1}. In this case we have  $0 < v_1 < 1/\sqrt{2} < v_2 < \sqrt{2}$ with $v_2 \in [2v_1, \sqrt{2}]$ and $A= v_1\sqrt{2}$, $B= v_2\sqrt{2}-1$ such that for given $A$ we have that $2A-1 \leq B \leq 1$. 
It turns out that we can calculate the expected discrepancy as a rational function of $A$ and $B$, which has a unique minimum.
For simplicity we restrict considerations to $1/2\leq A \leq 1$ 
%and $2A-1 \leq B \leq A$ 
in Lemma \ref{lem:n3}.

\begin{lemma} \label{lem:n3}
Let $1/(2 \sqrt{2}) \leq v_1 < 1/\sqrt{2}$ and $2v_1 \leq v_2 \leq \sqrt{2}-\frac{\sqrt{2}-2v_1}{2}$.
%and set $A=v_1 \sqrt{2}$, $B=v_2 \sqrt{2}-1$ such that for given $A$ we have $2A-1 \leq B \leq A$. 
Let $\bOmega_{v_1,v_2}^{(3)}$ be the corresponding partition of the unit square into three sets. Then
%$$ \min_{0 < v_1 < 1/\sqrt{2}} \min_{2v_1 \leq v_2 < ...}  \E \cL_2^2(\cP_{\bOmega_{v_1,v_2}^{(3)}} ) = \E \cL_2^2(\cP_{\bOmega_{v_1^{\ast},v_2^{\ast}}^{(3)}} ) =0.0268044\ldots,$$
$$ 
\min_{v_1,v_2}   \E \cL_2^2(\cP_{\bOmega_{v_1,v_2}^{(3)}} ) = \E \cL_2^2(\cP_{\bOmega_{v_1^{\ast},v_2^{\ast}}^{(3)}} ) =0.0267804\ldots,
$$
for $v_1^{\ast} = 0.5130\ldots$ and $v_2^{\ast} = 1.1249\ldots$
\end{lemma}
\begin{proof}
%\end{proof}
We set $A= v_1\sqrt{2}$, $B= v_2\sqrt{2}-1$ such that for given $1/2 \leq A \leq 1$ we have $2A-1 \leq B \leq A$.

We get $|\Omega_1|= A^2/2$, $|\Omega_3| = (1-B)^2/2$ and $|\Omega_2| = 1-|\Omega_1| - |\Omega_3|$. 
To calculate the expected discrepancy we subdivide the unit square into 6 sets, $S_1, \ldots, S_{6}$, as illustrated in Fig. \ref{fig:cases2} and we use the symmetry along the diagonal in cases II-VI.
For a point $(x,y) \in S_i$ we denote the expected discrepancy function by $f_i(x,y)$   and we calculate the expected discrepancy for the whole partition as
\begin{equation} \label{eqf_i}
\E \cL_2^2(\cP_{\bOmega_{v_1,v_2}^{(3)}} ) = \int_{S_1} f_1(x,y) \,\dd(x,y) + 2 \sum_{i=2}^6 \int_{S_i} f_i(x,y) \,\dd(x,y). 
\end{equation} 

For illustration, we give the calculations for the first case and refer to Appendix B for the other (equally elementary, but much more technical) cases.
\paragraph{Case I} Let $(x,y) \in \Omega_1=S_1$. Then $q_2(x,y)=q_3(x,y)=0$ and $q_1(x,y)=2xy/A^2$. Hence, we get
$$f_1=\E\left ( \frac{\#(x,y)}{3} - x y \right )^2 = \frac{x y(2 + 3 (-4 + 3A^2)x y)}{9 A^2},$$
and
$$ \int_0^{A} \int_{0}^{A-x} \frac{x y(2 + 3 (-4 + 3A^2)x y)}{9 A^2} \dd y \dd x = \frac{1}{540} A^2 (5-4A^2 + 3A^4). $$
\paragraph{Cases II - VI} With similar calculations as in Case I we can obtain expressions $f_2, \ldots, f_6$ for the expected value of the discrepancy function for points $(x,y)$ in each of the sets and in dependence of the parameters $A$ and $B$; see  Appendix B for details. Integrating these functions over their respective domains and summing the values, 
%i.e.,
%\begin{align*} 
%\int_{0}^{A}  \int_{0}^{A-x} f_1(x,y) dy dx &+ 2 \int_{}^{}  \int_{}^{} f_2(x,y) dy dx + 2 \int_{}^{}  \int_{}^{} f_3(x,y) dy dx \\ 
%&+ 2 \int_{}^{}  \int_{}^{} f_4(x,y) dy dx + 2 \int_{}^{}  \int_{}^{} f_5(x,y) dy dx + 2 \int_{}^{}  \int_{}^{} f_6(x,y) dy dx,
%\end{align*} 
gives us the following rational function in $A$ and $B$ which we can minimize over $1/2\leq A \leq 1$ and $2A-1 \leq B \leq A$ in order to obtain the two parameters $A$ and $B$ that generate the partition with the smallest expected discrepancy in this family; i.e. 

%\begin{align*}
%\E \cL_2^2(\cP_{\bOmega_{v_1,v_2}^{(3)}} ) =
%&\frac{1}{25920 A^2 (A^2+(B-2) B-1)} \cdot \\
%&\hspace{1cm} \left. (480 A^8+6912 A^7 (B+1)+16 A^6 (-527 + 330 B + 195 B^2) \right. \\
%&\hspace{1cm} \left. -384 A^5 (13 + 67 B + 87 B^2 + 33 B^3)+ \right. \\
%&\hspace{1cm}\left. 12 A^4 (829 + 1952 B + 2720 B^2 + 1228 B^3 + 431 B^4) \right. \\
%&\hspace{1cm}\left. +48 A^3 (-54 - 181 B - 204 B^2 + 23 B^3 + 120 B^4 + 24 B^5) \right. \\
% &\hspace{1cm}  \left. -4 A^2 (118 + 104 B + 345 B^2 + 776 B^3 - 728 B^4 + 1368 B^5 + 393 B^6)\right. \\
%&\hspace{1cm}   \left. +48 A (-10 + 7 B + 76 B^2 - 3 B^3 - 140 B^4 - 42 B^5 + 36 B^6 + 12 B^7)\right. \\
%&\hspace{1cm}\left. -(56 B + 774 B^2 + 556 B^3 - 1364 B^4 - 1728 B^5 - 70 B^6 + 396 B^7 + 
% 99 B^8)\right).
%\end{align*}

\begin{align*}
&\E \cL_2^2(\cP_{\bOmega_{v_1,v_2}^{(3)}} ) =
 \frac{1}{12960 A^2 (-1 + B)^2 (-1 + A^2 + (-2 + B) B)} \times \\
&\hspace{0.5cm} \left (128 B^7 - 64 B^8 + A^{10} (-1440 + 2880 B - 1440 B^2)  \right. \\
 &\hspace{0.5cm}+A^9 (2592 - 2592 B - 2592 B^2 + 2592 B^3)  \\
&\hspace{0.5cm}+ A^8 (-3832 + 3936 B - 1248 B^2 + 5760 B^3 - 4680 B^4)  \\
&\hspace{0.5cm}+ A^7 (10368 - 12288 B - 4032 B^2 + 8640 B^3 - 6336 B^4 + 4032 B^5)  \\
&\hspace{0.5cm}+ A^6 (-13936 + 19408 B + 4360 B^2 - 12864 B^3 - 1104 B^4 + 6480 B^5 - 
    3240 B^6)  \\
 &\hspace{0.5cm}+A^5 (5472 - 8544 B - 8736 B^2 + 20384 B^3 - 5280 B^4 - 2400 B^5 - 
    1440 B^6 + 1440 B^7)  \\
&\hspace{0.5cm} +A^4 (1908 - 1920 B + 5496 B^2 - 8976 B^3 - 480 B^4 + 5760 B^5 - 
    1608 B^6 - 144 B^7 - 36 B^8) \\
&\hspace{0.5cm}   +  A^3 (-480 - 1440 B + 1920 B^2 + 3840 B^3 - 7840 B^4 + 3104 B^5)  \\
&\hspace{0.5cm} +A^2 (-787 + 1740 B + 528 B^2 - 6292 B^3 + 9198 B^4 - 3492 B^5 \\
 &\hspace{1.5cm}   -212 B^6 + 84 B^7 + 237 B^8 - 72 B^9 - 36 B^{10})\\
 &\hspace{0.5cm} \left.  + A (-768 B^6 + 384 B^7) \right)
\end{align*}
This function can be minimised using a standard computer algebra system.
The minimum of this function is $0.0267804$ for the parameter values $A= 0.725501$ and $B= 0.590843$. 
These parameter values correspond to $v_1 = 0.5130\ldots$ and $v_2 = 1.1249\ldots$.
\end{proof}

%\textcolor{red}{we can not!! this does not satisfy the assumptions of the lemma!! this is case 2 in figure 7. add calculation for this case as another lemma}

%As a corollary, we can calculate the expected discrepancy of $\Omega_{\mathrm{ev}}^{(3)}$, which satisfies the assumptions of Lemma \ref{lem:n3}. In this case
%$$ v_1 = \sqrt{\frac{1}{3}}, \quad v_2 = \sqrt{2} - \sqrt{\frac{1}{3}} $$
%such that
%$$ A= \frac{\sqrt{2}}{\sqrt{3}} =0.816\ldots, \quad B= 2 - \frac{\sqrt{2}}{\sqrt{3}} -1 = 0.183\ldots $$
%Hence, we obtain from the function in the above proof that
%$$ \E \cL_2^2(\cP_{\Omega_{\mathrm{ev}}^{(3)}} ) = 0.0292035. $$

%\textcolor{red}{[[check what is acutally meant by B in the proof! is it the length or the location?]]}

In a similar fashion we can analyse the second case in Fig. \ref{fig:cases1}.

\begin{lemma} \label{lem:n3a}
Let $1/(2 \sqrt{2}) \leq v_1 < 1/\sqrt{2}$ and $\frac{1}{\sqrt{2}} \leq v_2 \leq 2 v_1$.
%and set $A=v_1 \sqrt{2}$, $B=v_2 \sqrt{2}-1$ such that for given $A$ we have $2A-1 \leq B \leq A$. 
Let $\bOmega_{v_1,v_2}^{(3)}$ be the corresponding partition of the unit square into three sets. Then
%$$ \min_{0 < v_1 < 1/\sqrt{2}} \min_{2v_1 \leq v_2 < ...}  \E \cL_2^2(\cP_{\bOmega_{v_1,v_2}^{(3)}} ) = \E \cL_2^2(\cP_{\bOmega_{v_1^{\ast},v_2^{\ast}}^{(3)}} ) =0.0268044\ldots,$$
$$\min_{v_1,v_2}  \E \cL_2^2(\cP_{\bOmega_{v_1,v_2}^{(3)}} ) = \E \cL_2^2(\cP_{\bOmega_{v_1^{\ast},v_2^{\ast}}^{(3)}} ) =0.0268054\ldots,$$
for $v_1^{\ast} = 0.5329\ldots$ and $v_2^{\ast} = 1.06582\ldots$
\end{lemma}

\begin{proof}
The proof follows the same lines as in Lemma \ref{lem:n3}; the only difference is the subdivision of the unit square as well as the range of $B$.
We set again $A=v_1 \sqrt{2}$, $B=v_2 \sqrt{2}-1$ such that for given $1/2 \leq A \leq 1$ we have $0 \leq B \leq 2A-1$.

We get $|\Omega_1|= A^2/2$, $|\Omega_3| = (1-B)^2/2$ and $|\Omega_2| = 1-|\Omega_1| - |\Omega_3|$. 
To calculate the expected discrepancy we subdivide the unit square into 6 sets, $S_I, \ldots, S_{VI}$, as illustrated in Fig. \ref{fig:cases2} and we use again the symmetry along the diagonal in cases II-VI.
For a point $(x,y) \in S_i$ we denote the expected discrepancy function  by $f_i(x,y)$ and we calculate the expected discrepancy for the whole partition as
$$ \E \cL_2^2(\cP_{\bOmega_{v_1,v_2}^{(3)}} ) = \int_{S_1} f_1(x,y) \,\dd(x,y)  + 2 \sum_{i=2}^6 \int_{S_i} f_i(x,y)\, \dd(x,y). $$

%For illustration we study Case IV in the following.

%\paragraph{Case IV} 

%\paragraph{In total}
As before, explicit expressions for $f_i$ on $S_i$, $i=1,\ldots,6$ can be obtained. They depend on the parameters $A$ and $B$. Integrating these functions over their respective domains and summing the values, 
%i.e.,
%\begin{align*} 
%\int_{0}^{A}  \int_{0}^{A-x} f_1(x,y) dy dx &+ 2 \int_{}^{}  \int_{}^{} f_2(x,y) dy dx + 2 \int_{}^{}  \int_{}^{} f_3(x,y) dy dx \\ 
%&+ 2 \int_{}^{}  \int_{}^{} f_4(x,y) dy dx + 2 \int_{}^{}  \int_{}^{} f_5(x,y) dy dx + 2 \int_{}^{}  \int_{}^{} f_6(x,y) dy dx,
%\end{align*} 
gives us again a rational function in $A$ and $B$ which we can minimize over $1/2\leq A \leq 1$ and $0 \leq B \leq 2A-1$ in order to obtain the two parameters $A$ and $B$ that generate the partition with the smallest expected discrepancy in this family. More explicitly, we have
\begin{align} \label{eq:lem5}
&\E \cL_2^2(\cP_{\bOmega_{v_1,v_2}^{(3)}} ) =
 \frac{1}{1620 A^2 (B-1)^2 \left(A^2+(B-2) B-1\right)} \times \\
\notag & \quad \left( A^8 \left(-6 B^2+12 B-14\right)+48 A^7 B+A^6 \left(-3 B^4+12 B^3+140 B^2-544 B+283\right) \right. \\
\notag & \quad +A^5 \left(208 B^3-672 B^2+1152 B-576\right) \\
\notag & \quad + A^4 \left(-3 B^6-18 B^5-15 B^4-420 B^3+1065 B^2-942 B+333\right) \\
\notag & \quad +A^3 \left(208 B^5-440 B^4+480 B^3-480  B^2+120\right) \\
\notag & \quad +A^2 \left(-6 B^8-24 B^7+134 B^6-444 B^5+870 B^4-704 B^3+264 B^2+168 B-146\right) \\
\notag & \quad \left. +A \left(48 B^7-96 B^6\right)-8 B^8+16 B^7 \right).
\end{align}
Using again a computer algebra system we obtain that the minimum of this function is $0.0268054$ for the parameter values $A= 0.753647$ and $B= 0.507294$. These parameter values correspond to $v_1 = 0.5329\ldots$ and $v_2 = 1.06582\ldots$.
\end{proof}

\begin{center}
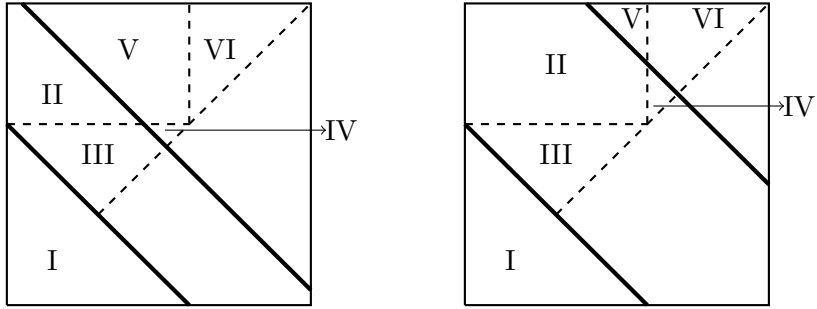
\begin{figure}[h!]

\begin{tikzpicture}[scale=4]

\draw [thick] (2,0) -- (3,0)--(3,1)--(2,1)--(2,0);
\draw [dashed, thick] (2.3,0.3) -- (3,1);
\draw [ultra thick] (2.05,1) -- (3,0.05);
\draw [ultra thick] (2,0.6) -- (2.6,0);

%\draw[dashed, thick] (2.7,0.7) -- (2.7,1);
\draw[dashed, thick] (2.6,0.6) -- (2.6,1);
\draw[dashed, thick] (2,0.6) -- (2.6,0.6);

%\draw[dotted] (2.4,0.6)--(2.4,1);
%\draw[dotted] (2.6,0)--(2.6,0.6);

\node at (2.15,0.15) {I};
\node at (2.3,0.5) {III};
\node at (2.15,0.7) {II};
\node at (3.1,0.58) {IV};
\node at (2.4,0.85) {V};
\node at (2.7,0.85) {VI};

\draw[<-] (3.05,0.58) -- (2.52,0.58);
\end{tikzpicture} \hspace{1cm}
\begin{tikzpicture}[scale=4]

\draw [thick] (2,0) -- (3,0)--(3,1)--(2,1)--(2,0);
\draw [dashed, thick] (2.3,0.3) -- (3,1);
\draw [ultra thick] (2.4,1) -- (3,0.4);
\draw [ultra thick] (2,0.6) -- (2.6,0);

%\draw[dashed, thick] (2.7,0.7) -- (2.7,1);
\draw[dashed, thick] (2.6,0.6) -- (2.6,1);
\draw[dashed, thick] (2,0.6) -- (2.6,0.6);

%\draw[dotted] (2.4,0.6)--(2.4,1);
%\draw[dotted] (2.6,0)--(2.6,0.6);

\node at (2.15,0.15) {I};
\node at (2.3,0.5) {III};
\node at (2.3,0.8) {II};
\node at (3.1,0.66) {IV};
\node at (2.55,0.95) {V};
\node at (2.8,0.95) {VI};

\draw[<-] (3.05,0.66) -- (2.62,0.66);
\end{tikzpicture}

\caption{{The subdivisions for the two middle cases in Fig. \ref{fig:cases1}. }} \label{fig:cases2}
%Middle: Minimal partition within family into three sets. Right: Partition into 4 sets that improves classical jittered sampling.
\end{figure}
\end{center}

The two lemmas provide two interesting insights. 
Firstly, we can now easily analyse the equivolume partition within this family.

\begin{corollary}
Let $ v_1 = \sqrt{\frac{1}{3}},$ and $v_2 = \sqrt{2} - \sqrt{\frac{1}{3}} $, then 
the mean $\cL_2$-discrepancy of $\bOmega_{\ast}^{(3)}= \bOmega_{v_1,v_{2}}^{(3)}$ is 
$$ \E \cL_2^2(\cP_{\Omega_{\ast}^{(3)}} ) = 0.0290077. $$
\end{corollary}
\begin{proof}
We have that $ A= \frac{\sqrt{2}}{\sqrt{3}} =0.816\ldots$ and $B= 1 - \frac{\sqrt{2}}{\sqrt{3}} = 0.183\ldots $.
Hence, we see that this case satisfies the assumptions of Lemma \ref{lem:n3a}. Using \eqref{eq:lem5}  we obtain the value.
\end{proof}

Secondly, combining the results of the two lemmas, we can now fix a parameter $A$ in $[1/2,1]$ and analyse all partitions for this fixed $A$ and any parameter $B$ in $[0,A]$. 
As it turns out, if we fix $A$ and plot the expected discrepancy as a function of the parameter $B$, then this function is very well behaved and has one unique minimum; see Fig. \ref{N3family}. 

\begin{remark} \label{rem1}
It is interesting to note that the minimal parameters in Lemma \ref{lem:n3} are not at the boundary of the two cases; i.e. for $A= 0.72550$ we have that $2A-1 = 0.451003 < 0.590843 = B < 0.72550 = A$. The expected discrepancy for the partition generated by $(A,2A-1)$ is 
$$0.0269763 $$
and is thus only slightly larger.
Furthermore, it turns out that the minimal parameters in Lemma \ref{lem:n3a} are exactly at the boundary; i.e. for $A= 0.753647$, we have that $2A-1=B= 0.507294$, whereas the minimum for this $A$ is obtained for $B^*=0.516474$ and the expected discrepancy is
$$ 0.0268046 $$
and is thus only slightly smaller.
\end{remark} 

Interestingly, the minimum for a given $A$ can be obtained in either of the two cases analysed in Lemma \ref{lem:n3} and \ref{lem:n3a} as can be seen from the examples. As a rule of thumb, the global minimum within this family is obtained for parameters $A$ and $B$ in which the minimum for fixed $A$ lies almost at the interval boundary, i.e. for which $B_{\min} \approx 2A-1$.%; see Fig. \ref{N3family2}.

This observation relates to Question \ref{qu1} of Section \ref{subsect:2.5}. It illustrates that the equivolume property appears to have no particular significance within this simple family of partitions. It rather seems that other geometric reasons drive the minimisation.
%In fact, it would not come to our attention, if it was not for the equivolume property.

\begin{figure}
\includegraphics[scale=0.4]{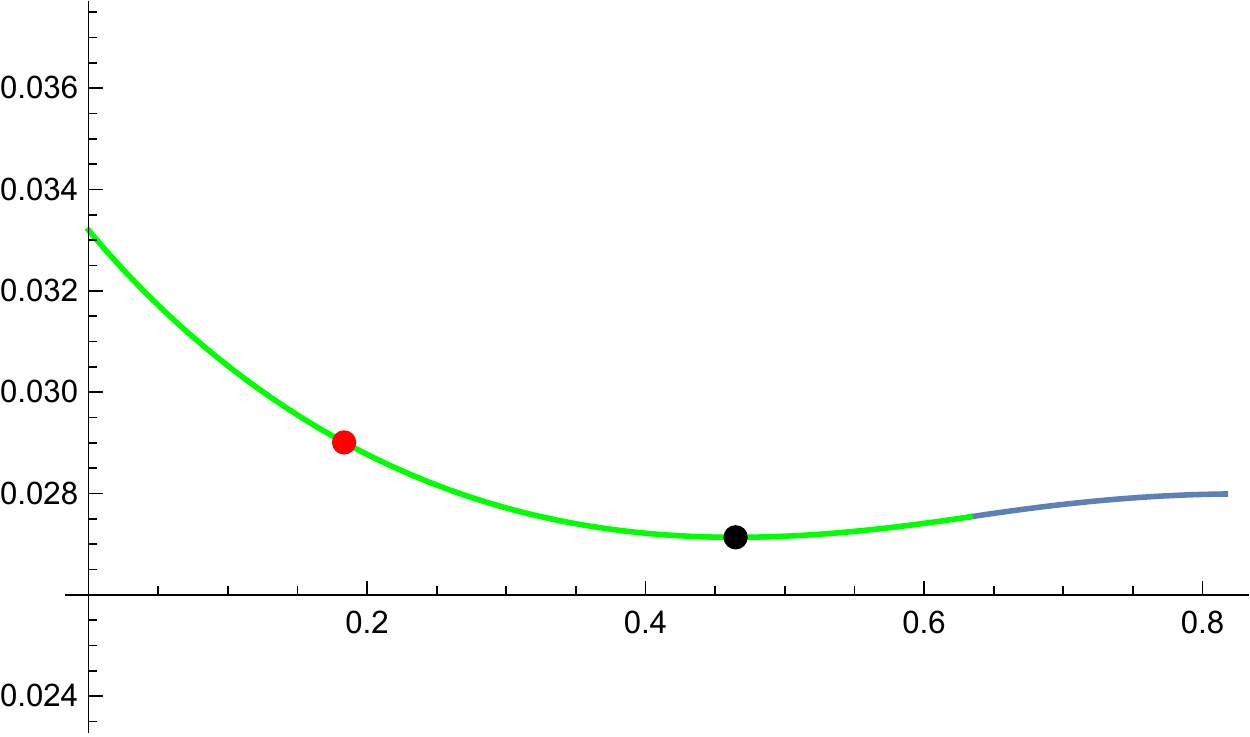}
\includegraphics[scale=0.4]{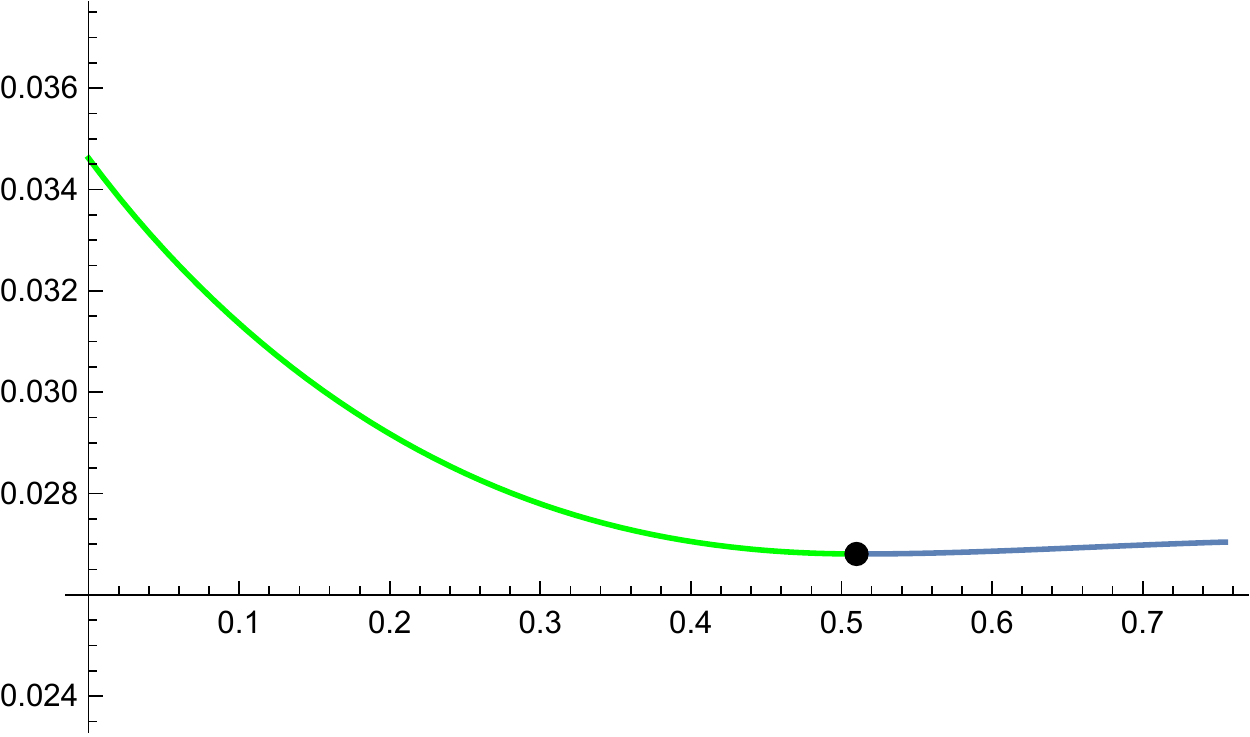}
\includegraphics[scale=0.4]{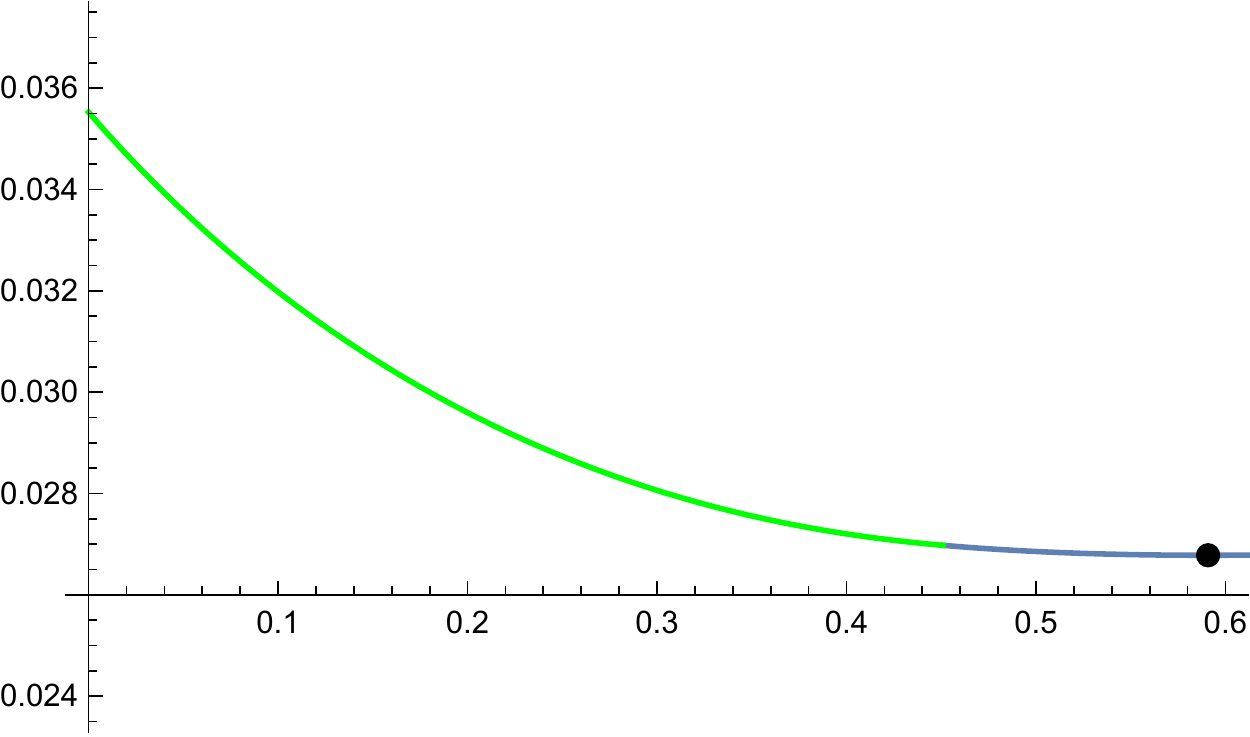}
\caption{The two colours illustrate the different parameter ranges for $B$ when $A$ is fixed.  The black dots represent the minima. Left: The graph for $A=\sqrt{2}/\sqrt{3}$. The red dot denotes the values for the equivolume partition. Middle: The graph for $A=0.755$. Right: The graph for $A=0.72550$. }
\label{N3family}
\end{figure}

%\begin{figure}
%\includegraphics[scale=0.4]{Bild3.pdf}
%\caption{The global picture. The $x$-axis contains the values of $A$. For each $A$, we plot the minimal expected discrepancy obtained for a $B$ in the first and second parameter region. The global minimum is obtained for the parameter pair $(A,B)$ for which the minima of the two cases are the same; i.e. the $A$ for which the two curves intersect.}
%\label{N3family2}
%\end{figure}

\begin{figure}
\includegraphics[scale=0.3]{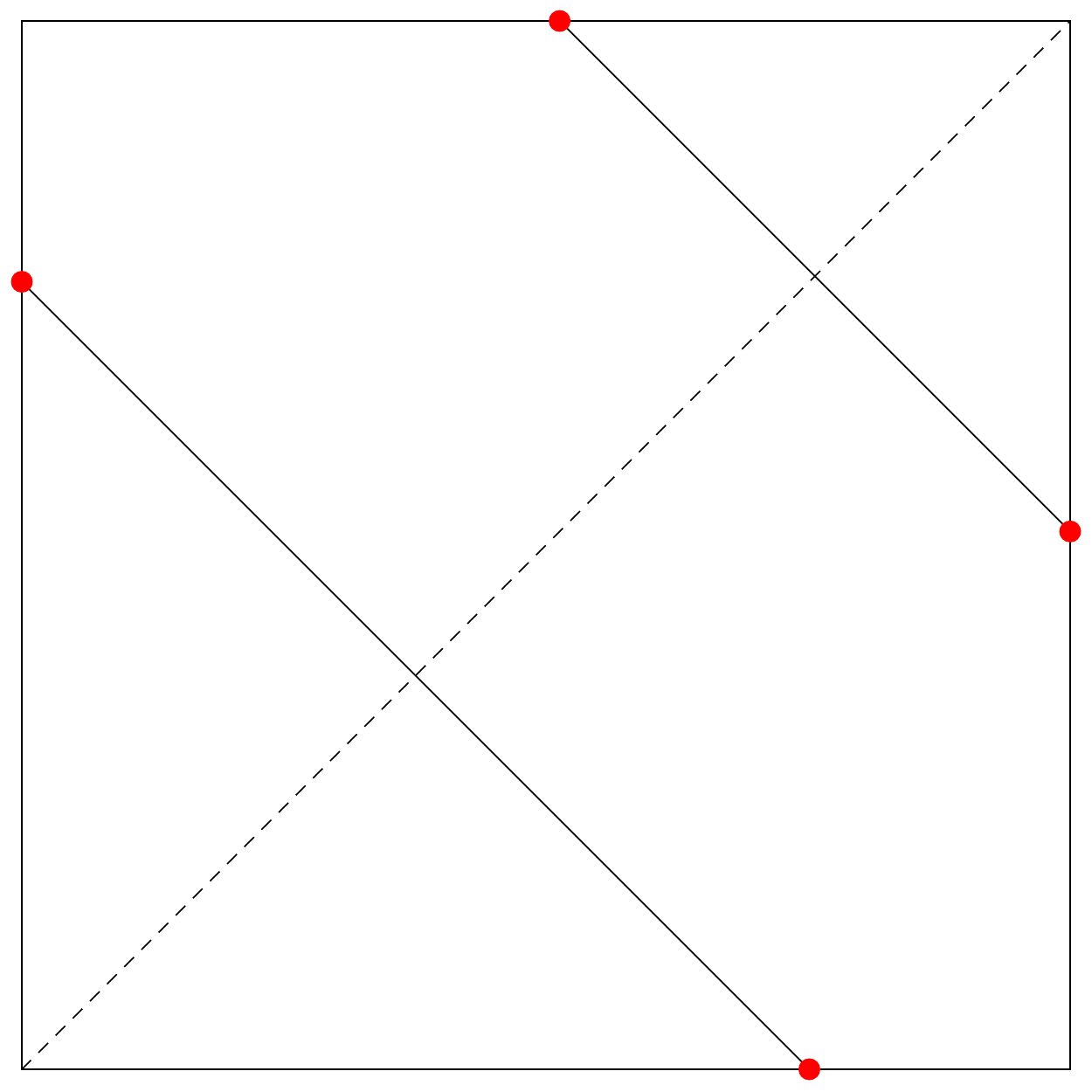} \hspace{3cm}
\includegraphics[scale=0.3]{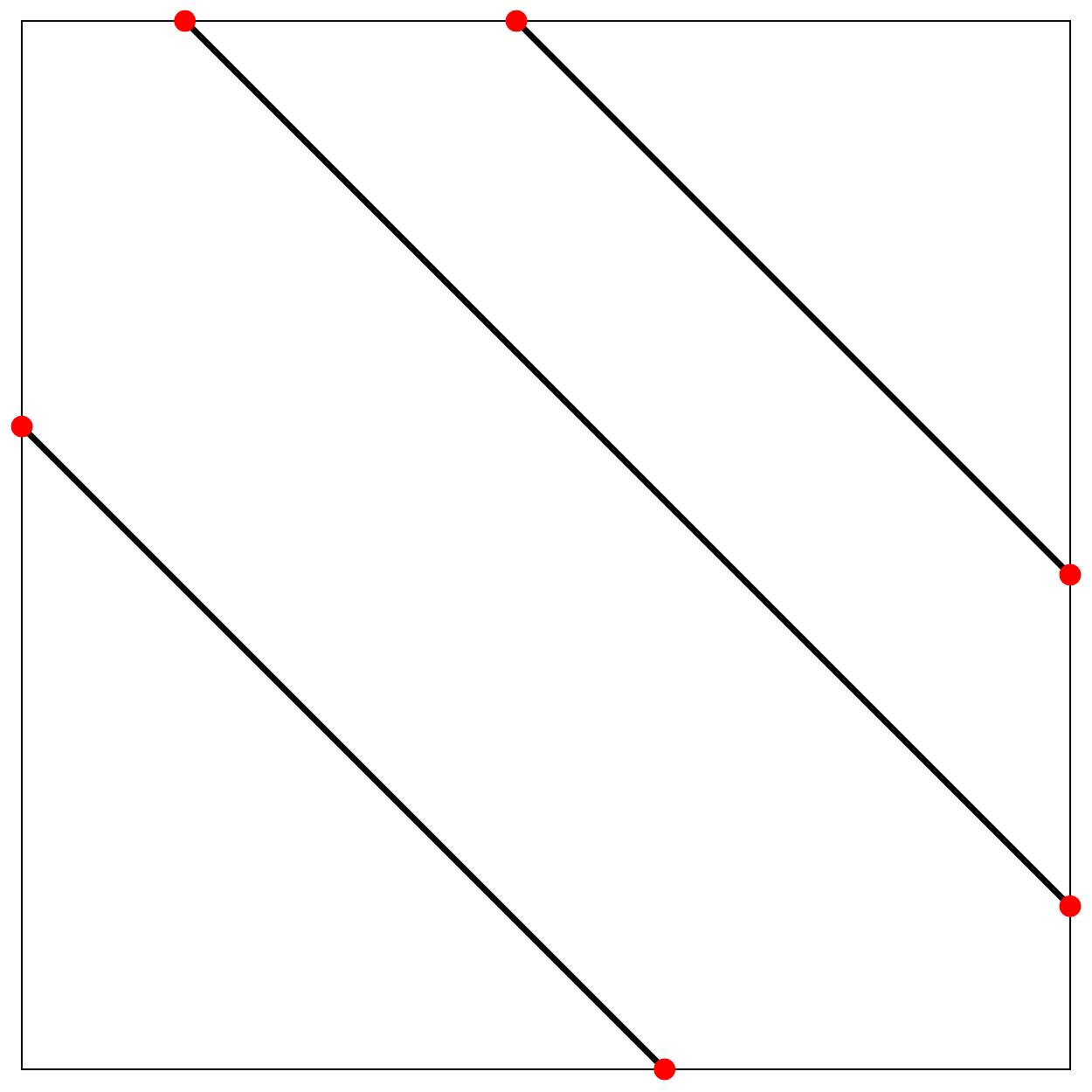}
\caption{Left: Minimal partition within the family $\bOmega_{\bv v}^{(3)}$ into three sets. Right: Partition into 4 sets that improves classical jittered sampling.}
\label{N3min}
\end{figure}

%\begin{figure}
%\includegraphics[scale=0.3]{n3partition.pdf} \hspace{3cm}
%\includegraphics[scale=0.3]{n4lines.pdf}
%\caption{Left: Minimal partition within family into three sets. Right: Partition into 4 sets that improves classical jittered sampling.}
%\label{N3min}
%\end{figure}

%%%%%%%%%%%%%%%%%%%
\subsection{An algorithmic approach}
In order to run systematic experiments within this family of partitions, we implemented an algorithm that takes as input an arbitrary vector $\mathbf{v}=\{v_1, \ldots, v_{N-1}\}$ with increasing entries in $[0,\sqrt 2]$ as well as a point $(x,y) \in [0,1]^2$ and outputs the expected value of the discrepancy function of the set of $N$ points generated from the partition $\bOmega_{\mathbf{v}}^{(N)}$ on the interval $[0,x] \times [0,y]$. This allows for an approximation of the expected value of the $\cL_2$-discrepancy of $\cP_{\bOmega_{\mathbf{v}}^{(N)}}$ using standard results from the theory of quasi-Monte Carlo integration. 

The algorithm is based on a simple geometric consideration. Assume $0 \leq y\leq x \leq 1$. As we have seen, 
we need to determine the probability 
$$q_i=q_i(x,y)= P(\bX_i \in \Omega_i \cap[0,x] \times [0,y]) = \frac{| \Omega_i \cap [0,x] \times [0,y] | }{| \Omega_i |}$$
with which the point $\bX_i \in \Omega_i$ lies in the box $[0,x] \times [0,y]$ for $1\leq i \leq N$. The expectation is then obtained from Proposition \ref{prop:3}. 
Hence, we first need to calculate the respective areas of the sets $\Omega_i$. To calculate their intersections with $[0,x] \times [0,y]$ we divide the set $v_1,\ldots,v_{N-1}$ into four subsets depending on which of the vertices of $[0,x] \times [0,y]$ are on the left or on the right of $\ell_1,\ldots,\ell_{N-1}$, respectively. More precisely:  the four lines with slope $-1$ through the vertices of $[0,x] \times [0,y]$ have distances $0=u_0\le u_1\le u_2\le u_3\le \sqrt2$ from the origin. The $j$th subset consists then of all $v_i$'s between $u_{j-1}$ and $u_j$ for $j=1,\ldots,4$, where  we have put $u_4=\sqrt 2$; see Fig. \ref{fig:algorithm} (left).  Different formulae are used to compute the intersection in each case.

This elementary algorithm leaves us with two conclusions. On the one hand, it is rather straightforward to calculate the expected value of the discrepancy function for a given box $[0,x] \times [0,y]$. On the other hand, it is incredibly tedious to do so. While this calculation can be solved algorithmically in a straightforward fashion, there is little hope to compute the expectation analytically since we have a different set of success probabilities for each box generated by a vector $(x,y)$.

\begin{center}
\begin{figure}[h!]
\begin{tikzpicture}[scale=0.8]
%%%%%%%%%%%%%%%%%%%%%%%
\draw[thick,fill=gray!25] (0,0) -- (0,2.3) -- (4.5,2.3) -- (4.5,0) -- (0,0);
\draw[thick] (0,0) -- (5,0)-- (5,5) -- (0,5)-- (0,0);
\draw[dotted] (0,0) -- (5,5);
\draw[dashed] (0,2.3) -- (2.3,0);
\draw[dashed] (4.5,0) -- (2.3,2.3);
\draw[dashed] (3.4,3.4) -- (4.5,2.3);
%\draw (0,2) -- (3,2) -- (3,0);

\draw[thick] (0,1) -- (1,0);
\draw[thick] (0,2) -- (2,0);
\draw[thick] (0,3) -- (3,0);
\draw[thick] (0,4) -- (4,0);
\draw[thick] (0,5) -- (5,0);
\draw[thick] (1,5) -- (5,1);
\draw[thick] (2,5) -- (5,2);
\draw[thick] (3,5) -- (5,3);
\draw[thick] (4,5) -- (5,4);

\node at (3.4,3.4) {$\bullet$}; 
\node at (2.25,2.25) {$\bullet$}; 
\node at (1.15,1.15) {$\bullet$}; 
{\draw[thin,<->] (-0.15,0.15) -- (1,1.30);
\node at (.45,1.05) {\footnotesize $u_1$}; }
{\draw[thin,<->] (0.15,-0.15) -- (3.55,3.25);
	\node at (2.15,1.4) {\footnotesize $u_3$}; }
%\node at (2.5,2.5) {$\bullet$}; 

\node at (-0.5,2.3) {\footnotesize $y$}; 
\node at (4.5,-0.5) {\footnotesize $x$}; 
%\node at (2,-0.3) {\footnotesize$m_2$}; 
%\node at (2.8,2.5) {\footnotesize$v_5$}; 

%%%%%%%%%%%%%%%%%%%%%%%%%%%%%%
\end{tikzpicture}
%\begin{center}
%\begin{figure}[h!]
\quad
\begin{tikzpicture}[scale=1.33]
%%%%%%%%%%%%%%%%%%%%%%%
\draw[thick,fill=gray!25] (0,0) -- (1,0)-- (0,1) -- (0,0);
\draw[thick] (0,0) -- (3,0)-- (3,3) -- (0,3)-- (0,0);
\draw[dashed] (0,0) -- (3,3);

\draw[thick] (0,1) -- (1,0);
\draw[thick] (0,2) -- (2,0);
\draw[thick] (0,3) -- (3,0);
\draw[thick] (1,3) -- (3,1);
\draw[thick] (2,3) -- (3,2);

\node at (1,0) {$\bullet$}; 
%\node at (1.5,0) {$\bullet$}; 
\node at (2,0) {$\bullet$}; 
%\node at (2.5,0) {$\bullet$}; 
\node at (3,0) {$\bullet$}; 
\node at (3,1) {$\bullet$}; 
\node at (3,2) {$\bullet$}; 
\node at (0,0) {$\bullet$}; 
\node at (3,3) {$\bullet$}; 

\node at (0.5,0.5) {$\bullet$}; 
\node at (1,1) {$\bullet$}; 
\node at (1.5,1.5) {$\bullet$}; 
\node at (2,2) {$\bullet$}; 
\node at (2.5,2.5) {$\bullet$};

\node at (1,-0.3) {\footnotesize $\ell_1$}; 
%\node at (1.5,-0.3) {\footnotesize $\ell_1^+$}; 
\node at (2,-0.3) {\footnotesize$\ell_2$}; 
%\node at (2.5,-0.3) {\footnotesize$\ell_2^+$}; 
\node at (3.3,-0.3) {\footnotesize$\ell_3$}; 
\node at (3.4,1) {\footnotesize$\ell_4$}; 
\node at (3.4,2) {\footnotesize$\ell_5$};

\node at (0.8,0.5) {\footnotesize $v_1$}; 
\node at (1.3,1) {\footnotesize$v_2$}; 
\node at (1.8,1.5) {\footnotesize$v_3$}; 
\node at (2.3,2) {\footnotesize$v_4$}; 
\node at (2.8,2.5) {\footnotesize$v_5$}; 

\node at (0.2,0.5) {\footnotesize$\Omega_1$}; 

%%%%%%%%%%%%%%%%%%%%%%%%%%%%%%
\end{tikzpicture}
\quad
\begin{tikzpicture}[scale=1.33]
%%%%%%%%%%%%%%%%%%%%%%%
\draw[thick,fill=gray!25] (0,0) -- (1.73,0)-- (0,1.73) -- (0,0);
\draw[thick] (0,0) -- (3,0)-- (3,3) -- (0,3)-- (0,0);
\draw[dashed] (0,0) -- (3,3);

\draw[thick] (0,1.73) -- (1.73,0);
\draw[thick] (0,2.45) -- (2.45,0);
\draw[thick] (0,3) -- (3,0);
\draw[thick] (3-1.73,3) -- (3,3-1.73);
\draw[thick] (3-2.45,3) -- (3,3-2.45);

\node at (1.73,0) {$\bullet$}; 
%\node at (1.5,0) {$\bullet$}; 
\node at (2.45,0) {$\bullet$}; 
%\node at (2.5,0) {$\bullet$}; 
\node at (3,0) {$\bullet$}; 
\node at (3,3-1.73) {$\bullet$}; 
\node at (3,3-2.45) {$\bullet$}; 
\node at (0,0) {$\bullet$}; 
\node at (3,3) {$\bullet$}; 

\node at (0.86,0.86) {$\bullet$}; 
\node at (1.22,1.22) {$\bullet$}; 
\node at (1.5,1.5) {$\bullet$}; 
\node at (1.77,1.77) {$\bullet$}; 
\node at (2.13,2.13) {$\bullet$};

\node at (1.73,-0.3) {\footnotesize $\ell_1$}; 
%\node at (1.5,-0.3) {\footnotesize $\ell_1^+$}; 
\node at (2.45,-0.3) {\footnotesize$\ell_2$}; 
%\node at (2.5,-0.3) {\footnotesize$\ell_2^+$}; 
\node at (3.3,-0.3) {\footnotesize$\ell_3$}; 
\node at (3.4,3-1.73) {\footnotesize$\ell_4$}; 
\node at (3.4,3-2.45) {\footnotesize$\ell_5$};

\node at (1.2,0.86) {\footnotesize $v_1$}; 
\node at (1.52,1.22) {\footnotesize$v_2$}; 
\node at (1.8,1.5) {\footnotesize$v_3$}; 
\node at (2.07,1.77) {\footnotesize$v_4$}; 
\node at (2.43,2.13) {\footnotesize$v_5$}; 

\node at (0.4,0.8) {\footnotesize$\Omega_1$}; 

%%%%%%%%%%%%%%%%%%%%%%%%%%%%%%
\end{tikzpicture}
%\caption{Middle: Illustration of equidistant partitions. Right: Illustration of equivolume partitions.} \label{fig:specialCases}
%\end{figure}
%\end{center}

\caption{Left: Illustration of the algorithm. The three bullet points indicate the projections of the vertices of 
		$[0,x] \times [0,y]$ on the main diagonal giving rise to the numbers $u_1,u_2,u_3$ and the division of $\{v_1, \ldots, v_{N-1}\}$ into four subsets. Middle: Illustration of the equidistant partition $\bOmega^{(6)}_{\bv v}$ with $\bv v= \frac{\sqrt{2}}{6} (1,\ldots,5)$. Right: Illustration of the equivolume partition $\bOmega^{(6)}_*$.} \label{fig:algorithm}
\end{figure}
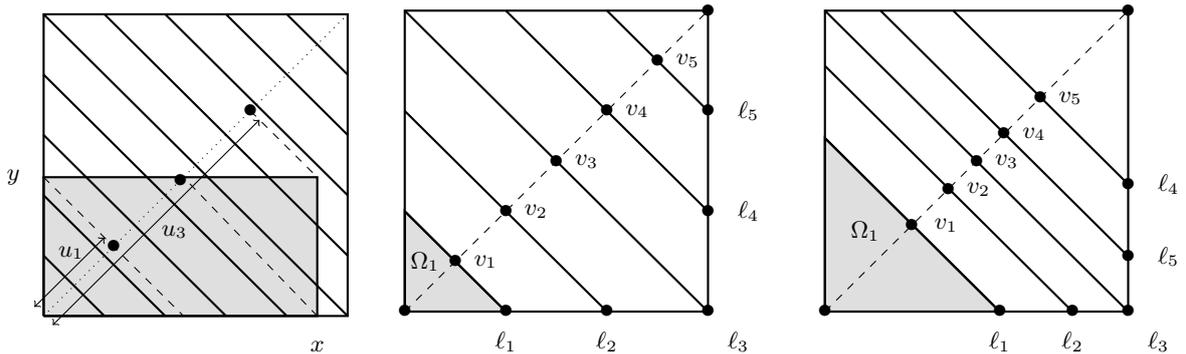
\end{center}

%%%%%%%%%%%%%%%%%%%
\subsection{Numerical results} \label{sec4:num}
In this final section, we present the results of three different sets of experiments. In the first two experiments we generate many instances of stratified point sets for a given fixed partition, calculate the \emph{$\cL_2$-discrepancy} of each point set and approximate the expected discrepancy of the partition by the mean of the experiment.
In the final experiment, we calculate and compare the \emph{star discrepancy} of different point sets.

We use Warnock's formula \cite{warnock} as presented in \cite[Proposition 2.15]{DP} to calculate the $\cL_2$-discrepancy of a given point set. That is, for any point set $\cP=\{\bv x_0, \ldots, \bv {x}_{N-1} \} \in [0,1]^d$ we have
\begin{equation} \label{warnock}
\cL_2(\cP)^2 = \frac{1}{3^d} - \frac{2}{N} \sum_{n=0}^{N-1} \prod_{i=0}^d \frac{1-x_{n,i}^2}{2} + \frac{1}{N^2} \sum_{m,n=0}^{N-1} \prod_{i=0}^d \min(1-x_{m,i}, 1-x_{n,i}),
\end{equation}
in which $x_{n,i}$ is the $i$-th component of the $\bv {x}_n$. 
We refer to \cite{heinrich1, heinrich2} for quick implementations of this formula.

First, we present a numerical observation which could, in principle, be proven along the same lines as Lemma \ref{lem:n3}. However, given the number of case distinctions such an analysis -- based on our elementary method -- would require, we only provide numerical evidence and state the result as a conjecture.

\begin{conjecture}
There exists $\mathbf{v}=(v_1,v_2,v_3)$ with $v_1< v_2< v_3$ in $[0,\sqrt{2}]$  such that 
$$\E \cL_2^2(\cP_{\bOmega_{\mathbf{v}}^{(4)} }) <  \E {\cL_2^2}(\cP_{\mathrm{jit4}}) = 0.01909\ldots$$
\end{conjecture}

We obtained various instances of partitions that seem to improve the classical jittered sampling by perturbing the three values of the vector $(\sqrt2/4)(1,2,3)$. We obtained the best numerical results for the three points
$$v_1^{\ast}=\frac{\sqrt{2}}{4} + 0.08, \quad v_2^{\ast}=\frac{\sqrt{2}}{2} + 0.11, \quad v_3^{\ast}=\frac{3\sqrt{2}}{4} - 0.02;$$
see Fig.~\ref{N3min}.
We simulated $10^6$ instances of stratified sets for this particular partition and calculated the discrepancy in each case with the formula of Warnock. Independently, we used our algorithm to estimate the expected discrepancy using $10^4$ many grid points. 
Both methods indicate that the first digits after the decimal points of the expected discrepancy are
$$ 0.0188\ldots,$$
which would be clearly better than the mean discrepancy of jittered sampling. 
%However, in both cases we can only guarantee that the numerical error of our methods is $\mathcal{O}(10^{-3})$ which is not enough to rigorously prove the conjecture. \textcolor{red}{[Markus: we should check whether this is really the case. Maybe some statistical magic actually makes our sample size large enough.]}

Next, we use Warnock's formula to empirically study the discrepancy of different point sets and constructions; i.e. for given $N$ we generate 500 samples and calculate the $\mathcal{L}_2$-discrepancy for each of these samples using Warnock's formula. The empirical mean of this sample approximates the expected value of the discrepancy.
We collect our numerical results in Table \ref{table1}.
%These results leave us with two conclusions. First, comparing the empirical estimate for the vertical strip partitions with the actual values that we calculated in Section \ref{} indicate that our numerical results are quite accurate. Secondly, 
Our numerical results suggest that the expected discrepancy of partitions $\bOmega_{\ast}^{(N)}$ is about a factor 2 smaller than the expected discrepancy of a set of random points:
\begin{conjecture} We conjecture that
$$ \lim_{N \rightarrow \infty} \frac{ \E {\cL_2}(\cP_N)^2}{\E {\cL_2}(\cP_{\bOmega_{\ast}^{(N)}})^2} = 2.$$
%$ \E {\cL_2}(\cP_N)^2 \sim 2\E {\cL_2}(\cP_{\bOmega_{\ast}^{(N)}})^2$ in which $\sim$ means in the highest order of $N$.
\end{conjecture}

\begin{table}[h]
\begin{center}
%\begin{tabular}{|c|| l || l | l || l || l |}
\begin{tabular}{|c|| l || l || l || l |}
\hline
%&random&\multicolumn{2}{l ||}{vertical lines} & equivol lines & jittered\\
%&random&vertical lines& equivol lines & jittered\\
$ \E {\cL_2}( \cdot )$ &$\cP_N$& $\cP_{\mathrm{vert}}$& $\cP_{\bOmega_{\ast}^{(N)}}$ & $\cP_{\mathrm{jit}}$\\
\hline
$N$ &  & & \text{empirical} & \text{empirical}\\
\hline
50 &0.00277778&  0.00168889  &0.00137637 & \\
100 &0.00138889& 0.000838889 &0.000699558 & 0.000163637\\
150 &0.000925926& 0.000558025 &0.000471159 & \\
\hline
200 &0.000694444&0.000418056 &0.000356743 & \\
%250 &0.000555556&0.000334222 &0.000272358 & \\
256 &0.000542535 & 0.000326369& 0.000269319& 0.0000403301\\
300 &0.000462963&0.000278395 &0.000228231 & \\
\hline
350&0.000396825& 0.000238549& 0.000201676& \\
400&0.000347222& 0.000208681& 0.000172704 & 0.0000206345\\
450&0.000308642& 0.00018546& 0.000159365& \\
\hline
\end{tabular}
\medskip 

\end{center}
\caption{Expected $\cL_2$-discrepancy of different point sets, in which $N$ stands for the number of points. The empirical values are calculated as the mean of the discrepancy of 500 samples. We calculated the discrepancy of  individual samples with Warnock's formula. 
%\textcolor{red}{conjecture: equi vol lines is a factor 1/2 better than monte carlo [[double check with paper of pyke and van zwet! they also get a factor of 1/2 improvement!]]} 
}
\label{table1}
\end{table}

In our final experiment, we use an implementation of the Dobkin-Eppstein-Mitchell algorithm \cite{dob}  for the computation of the star discrepancy which was provided by Magnus Wahlstr\"{o}m; for details on the implementation we refer to \cite{doerr}.
This experiment relates to our comments in Section \ref{sect:extra} and shows that our partitions also seem to generate point sets that have a smaller expected star discrepancy than sets of $N$ i.i.d. uniformly random points. We leave the generalisation of the partitions $\bOmega_{\ast}^{(N)}$ for future research. It is in principle straightforward, but a bit technical and thus beyond the scope of this final proof-of-concept numerical experiment.

\begin{table}[h]
\begin{center}
%\begin{tabular}{|c|| l || l | l || l || l |}
\begin{tabular}{|c|  c |c|c|c|c|}
\hline
$ \E D^*( \cdot )$& &$\cP_N$& $\cP_{\mathrm{vert}}$& $\cP_{\bOmega_{\ast}^{(N)}}$ & $\cP_{\mathrm{jit}}$\\
%\hline
%\multicolumn{5}{| l |}{d=2} \\
\hline
%&&  &\text{empirical} & \text{empirical} & \text{empirical}\\
\hline
d=2 & $N$& & & & \\
\hline
&$10^2=100$ & 0.1129 & 0.1016 & 0.0975 & 0.0616 \\
&$32^2=1024$ & 0.0379 & 0.0316 & 0.0293 & 0.0127 \\
\hline
d=3 & & & & & \\
\hline
&$5^3=125$ & 0.1430 & 0.1233 & -- & 0.0910 \\
&$10^3=1000$ & 0.0483 & 0.0397 & -- & 0.0274 \\
\hline
d=5 & & & & & \\
\hline
&$4^5=1024$ & 0.0610 & 0.0560 & -- & 0.0463 \\
\hline
\end{tabular}
\medskip 

\end{center}
\caption{Mean star discrepancy of 20 experiments with different point sets in dimensions $d=2,3,5$.}
\label{table:discrepancy}
\end{table}

%%%%%%%%%%%%%%%%%%%
% References
%%%%%%%%%%%%%%%%%%%

\addresseshere

%\section*{Appendix -- Conventions}

%\begin{enumerate}
%\item Monte Carlo
%\item Fig. 
%\item $\dd\bv x$ and $\dd x$
%\end{enumerate}

\newpage
\section*{Appendix}

\subsection*{Appendix A -- Uniform distribution and equidistribution of partitions}%\label{appendixUifDistr}

The well-known fact that a sequence of Monte Carlo samples $(\cP_N)$ is almost surely uniformly distributed follows 
directly from the strong law of large numbers. A corresponding statement for the sequences in Definition \ref{def1} 
is based on the strong law of large numbers for triangular arrays. 
Note that it  does not require that the sampling points for different partitions $\bOmega^{(N)}$ are independent. This is why we do not introduce this assumption, although it is typically satisfied in the applications we have in mind.

\begin{proposition}\label{prop:2}
	Consider a sequence of partitions $\{\bOmega^{(N)}\}_{N\geq 1}$ of $[0,1]^d$, with $\bOmega^{(N)}=(\Omega_1^{(N)},\ldots,$ $\Omega_N^{(N)})$ consisting of Lebesgue-sets with positive content and let $\bX^{(N)}=(\bX_1^{(N)},\ldots, \bX_N^{(N)})$ be the vector of stratified sampling points based on $\bOmega^{(N)}$. 
	Then the triangular array $\widehat \bX=\big(\bX^{(N)}\big)_{N\in \NN}$ is almost surely uniformly distributed if and only if 
	\begin{align}\label{eqUnifDist}
		\lim_{N\to \infty}\sum_{i=1}^N \frac1{N|\Omega_i^{(N)}|}|\Omega_i^{(N)}\cap [{\bf x}, {\bf y}[|
		=|[{\bf x}, {\bf y}[| 
	\end{align}
	for all cubes $[{\bf x}, {\bf y}[ \subset [0,1]^d$.  
\end{proposition}
\begin{proof}
	
	For ${\bf x}, {\bf y}\in [0,1]^d$ with each component of $\by$ at least as large as the corresponding component of $\bx$, consider the  stochastic variables 
	$Y_{i}^{(N)}=1_{[{\bf x}, {\bf y}[}(\bX^{(N)}_i)$, where $\bX^{(N)}_1,\ldots,\bX_N^{(N)}$ is the stratified sample based on the partition $\bOmega^{(N)}$. The $Y$'s  are row-wise independent random variables with uniformly bounded variances and we have
	\[
	a_N(\bx,\by):=\frac1N\sum_{i=1}^N\E Y_i^{(N)}=\sum_{i=1}^N \frac1{N|\Omega_i^{(N)}|}\left|\Omega_i^{(N)}\cap [{\bf x}, {\bf y}[\right|.
	\]
	Therefore, the strong law of large numbers for triangular arrays (see 
	\cite[Theorem 2.2]{HuTaylor} with $p=1$, $\psi(t)=t^2$ and $a_n=n$) implies 
	\begin{equation}\label{eqSLLN}
		\frac{\# \big(\cP_{\bOmega^{(N)}}\cap [{\bf x}, {\bf y}[\big)}N-a_N(\bx,\by) \to 0
	\end{equation}
	almost surely as $N\to\infty$. 
	
	If assumption \eqref{eqUnifDist} holds, we have $\lim_{N\to\infty} a_N(\bx,\by)=|[{\bf x}, {\bf y}[|$, so \eqref{eqSLLN} implies  \eqref{eq:Def1} for almost every realization. Hence, the stratified sample points form almost surely a uniformly distributed triangular array. 
	
	If \eqref{eqUnifDist}  is violated, there must be ${\bf x}, {\bf y}\in [0,1]^d$ such that the limit in \eqref{eqUnifDist} does not exist or is different from $|[{\bf x}, {\bf y}[|$. In any case, there is a subsequence $\big(a_{N'}(\bx,\by)\big)$ of $\big(a_N(\bx,\by)\big)$ such that $a_{N'}(\bx,\by)\to a\ne |[{\bf x}, {\bf y}[|$ as $N'\to\infty$, and  \eqref{eqSLLN} shows that \eqref{eq:Def1} cannot hold for almost every realization. 
	
	Concluding,  the triangular array is uniformly distributed if and only if  \eqref{eqUnifDist} holds 
	for all rectangular sets $[{\bf x}, {\bf y}[\in [0,1]^d$. 
\end{proof}

In particular,  if $\{\bOmega^{(N)}\}_{N\geq 1}$ is a sequence of finite partitions of the unit cube such that all partitions are equivolume, then $|\Omega_i^{(N)}|=1/N$, so \eqref{eqUnifDist} is satisfied even without taking the limit. As a consequence, sequences of equivolume partitions are uniformly distributed; this is implication (d) in Fig.~\ref{Fig:Implic}. \medskip 
\begin{figure}[bh]
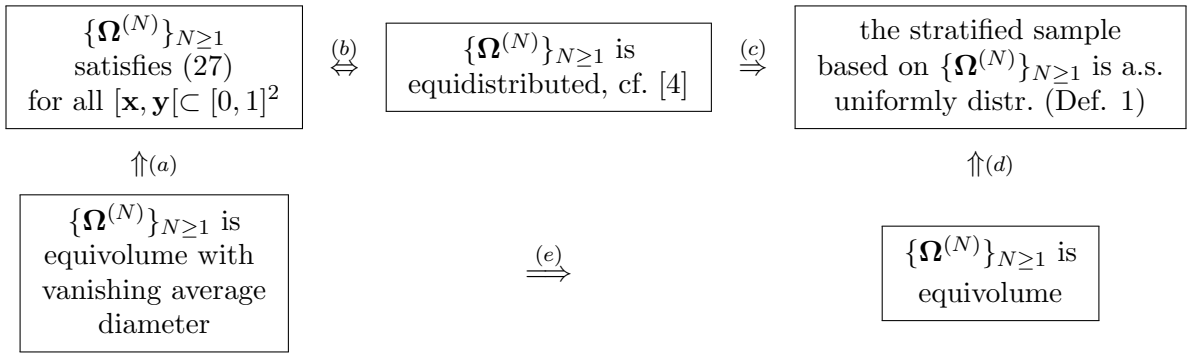
 
	\begin{tabular}{ccccc}
		
			\fbox{$
				\begin{array}{c}
					\{\bOmega^{(N)}\}_{N\geq 1}
					\\
					\text{satisfies \eqref{new}}
					\\
					\text{for all $[\bx,\by[\subset [0,1]^2$}
				\end{array}
				$}
		&
		$\stackrel{(b)}{\Leftrightarrow}$
		& 
		\fbox{$
			\begin{array}{c}
				\{\bOmega^{(N)}\}_{N\geq 1} \text{ is}
				\\
				\text{equidistributed, cf.~\cite{volcic2}}
			\end{array}
			$}
		&
		$\stackrel{(c)}{\Rightarrow}$
		&
		\fbox{$
			\begin{array}{c}
				\text{the stratified sample}\\
				\text{based on } \{ \bOmega^{(N)}\}_{N\geq 1} \text{ is a.s.}
				\\
				\text{uniformly distr.~(Def. \ref{def1})}
			\end{array}
			$}
		\\[5ex]
		$\Uparrow \scriptstyle (a)$
		&&
		&&$\Uparrow \scriptstyle (d)$
		\\[1ex]
			\fbox{$
				\begin{array}{c}
					\{\bOmega^{(N)}\}_{N\geq 1}\text{  is}
					\\
					\text{equivolume with }
					\\
					\text{vanishing average}
					\\
					\text{diameter}
				\end{array}
				$}
		&&
		$\stackrel{(e)}{\Longrightarrow}$
		&&
		\fbox{$
			\begin{array}{c}
				\{\bOmega^{(N)}\}_{N\geq 1} \text{ is}
				\\[.5ex]
				\text{equivolume}
			\end{array}
			$}
	\end{tabular}
	\caption{Different properties of sequences of partitions and their connections.}
	\label{Fig:Implic}	
\end{figure}

We conclude this section with a comparison of Definition \ref{def1} with the notion of equidistributed partitions from  \cite{volcic2}, which we now recall using our notation.  
A sequence $\{\bOmega^{(N)}\}_{N\geq 1}$ of finite partitions of the unit cube is called \emph{equidistributed} 
if for \emph{any} choice of  $\bx_i^{(N)}\in \Omega_i^{(N)}$, $i=1,\ldots,N$, the triangular array $\widehat{\bx}=\big(\bx_1^{(N)},\ldots,\bx_N^{(N)}\big)_{N\in \NN}$ is uniformly distributed (Definition \ref{def1}). Actually, in \cite{volcic2} this notion is introduced and exploited in larger generality, replacing the unit cube with a general separable metric space.

In Fig.~\ref{Fig:Implic}, we outline the connections between different notions for partitions, where a sequence $\{\bOmega^{(N)}\}_{N\geq 1}$ of finite partitions of $[0,1]^d$ is said to have \emph{vanishing average diameter} if the average diameter 
	\[\overline \delta_N=\frac{1}{N}\sum_{i=1}^N \text{diam}\,(\Omega_i^{(N)})\]
	converges to zero as $N\to \infty$. We also use the set 
	\[
	I_B(N)=\big\{i\in \{1,\ldots,N\}: \Omega_i^{(N)}\subset B\big\}, 
	\]
	of all indices $i$ for which $\Omega_i^{(N)}$ is completely contained in $B\subset [0,1]^2$, and the condition 
	\begin{equation}\label{new}
		\lim_{N\to \infty}\frac1N \#I_{[{\bf x}, {\bf y}[}(N)
		=|[{\bf x}, {\bf y}[|,\quad 	\text{ for all }[{\bf x}, {\bf y}[\subset [0,1]^d,
	\end{equation}
	stating that asymptotically the `correct' proportion of sets are completely contained in  $[{\bf x}, {\bf y}[$.  
	That an equi\-volume sequence of partitions with vanishing \emph{maximal} diameter is equidistributed was shown in 
	\cite[Lemma 1]{volcic2}, but our implication Fig.~\ref{Fig:Implic}(a) is stronger, as it allows for `large' or `elongated' sets, as long as the proportion of such sets within the $N$ sets of the partitions goes to zero as $N\to\infty$. 

	\begin{proposition}
		The implications in Fig.~\ref{Fig:Implic} hold. Apart from the equivalence in (b) none of the implications can be reversed.	
\end{proposition}
\begin{proof}
	To show (a) fix $[{\bf x}, {\bf y}[ \subset [0,1]^d$,
		assume that the sets of $\{\bOmega^{(N)}\}_{N\geq 1}$ are equivolume, and observe that 
		in this case 
		\[
		\frac{1}{N} \#I_{[{\bf x}, {\bf y}[}(N)=\left| \bigcup_{i\in I_{
				[{\bf x}, {\bf y}[}(N)}\Omega^{(N)}_i\right|\le 
		\left| \bigcup_{i=1}^N(\Omega^{(N)}_i\cap [{\bf x}, {\bf y}[)\right|=|[{\bf x}, {\bf y}[|. 
		\]
		On the other hand, if 
		\[
		T_B(N)=\big\{i\in \{1,\ldots,N\}: \Omega_i^{(N)}\cap B\ne \emptyset,\Omega_i^{(N)}\cap B^c\ne \emptyset\big\},
		\]
		describes the sets of the partition hitting both,  
		$B\subset [0,1]^d$ and its complement, we have 
		\[
		\frac{1}{N} \#I_{[{\bf x}, {\bf y}[}(N)+\frac{1}{N} \#T_{[{\bf x}, {\bf y}[}(N)
		=\left| \bigcup_{i\in I_{[{\bf x}, {\bf y}[}(N)\cup T_{[{\bf x}, {\bf y}[}(N)}\Omega^{(N)}_i\right|
		\ge |[{\bf x}, {\bf y}[|,
		\]
		 so it is enough to show that 
		\begin{align}\label{eqT}
			\frac{1}{N} \#T_{[{\bf x}, {\bf y}[}(N)\to 0
		\end{align}
		as $N\to \infty$.  To show \eqref{eqT} let $\varepsilon>0$ and $\alpha>0$ be given. 
		Markov's inequality and the assumption of vanishing average diameter yield the existence of  $N_0\in \N$ such that 
		\[
		\frac{1}N \#\{i:\text{diam}\,(\Omega_i^{(N)})>\alpha\}\le \frac{\overline \delta_N}{\alpha}\le     \frac{\varepsilon}2
		\]
		for all $N\ge N_0$. Hence, using again the equivolume property,  
		\[
		\frac{1}{N} \#T_{[{\bf x}, {\bf y}[}(N)\le \frac{1}N \#\{i:\text{diam}\,(\Omega_i^{(N)})>\alpha\}+
		\left|(\bd [{\bf x}, {\bf y}[)_\alpha\cap [0,1]^d\right|\le \frac{\varepsilon}2+4d\alpha,
		\]
		where $(\bd [{\bf x}, {\bf y}[)_\alpha$ is the set of all points  with distance at most $\alpha$ from the boundary of $[{\bf x}, {\bf y}[$. Choosing $\alpha=\varepsilon/(8d)$ implies \eqref{eqT}.  

	To show the equivalence in (b) we note that \eqref{new} implies \eqref{eqT}. In fact, if  $[{\bf x}, {\bf y}[ \subset [0,1]^d$ is fixed, the set $B=[0,1[^d\setminus [{\bf x}, {\bf y}[$ can be written as disjoint union of at most $k=3^d-1$ half-open rectangles $B_1,\ldots,B_k$, so 
		\[
		\frac{1}{N} \#T_{[{\bf x}, {\bf y}[}(N)\le \frac{1}{N} (N-\#I_{[{\bf x}, {\bf y}[}(N)-\sum_{j=1}^k \#I_{B_j}(N))\to 0,
		\]
		by applying   \eqref{new} to all $3^d$ rectangles.
	
		If $\cP_N=\{\bx_1,\ldots,\bx_N\}$ satisfies $\bx_i\in \Omega_i^{(N)}$ but is otherwise arbitrary,  then 
		\begin{align}\label{eqreview}
			I_{[\bx,\by]}(N)\le \#\{i: \bx_i\in [0,\bx]\}\le I_{[\bx,\by]}(N)+T_{[\bx,\by]}(N). 
		\end{align}
		and there are sets for which either of the two inequalities becomes an equality. 
		One can thus conclude that \eqref{new} (together with its consequence, \eqref{eqT}) implies equidistribution of $\{\bOmega^{(N)}\}_{N\geq 1}$. Conversely, assuming equidistribution, we may choose $\cP_N$ such that there is equality on the left side of \eqref{eqreview} for every $N$, and  \eqref{new}  follows.

	Implication (c) is trivial as a stratified sample point lies a.s. in the partition set it is associated to.
	Implication (d) has already been shown and (e) is trivial. 
	\medskip 
	
	None of the above implications can be reversed.That implication (a) cannot be reversed is clear, as one just takes a sequence satifying \eqref{new} and changes two sets appropriately to assure that the sequence is not equivolume. Implication (d) can be treated in a similar way. That implication (c) cannot be reversed can be seen by means of vertical strip partitions $(\bOmega^{(N)}_{\mathrm{vert}})$ in
	Fig.~ \ref{fig:simplePartition} (left). For all $N$ these partitions are equivolume and thus the stratified sample based on them is a.s.~uniformly distributed. However, this sequence of partitions is not equidistributed, as the limit in \eqref{new} is always zero. This example also shows that implication (e) cannot be reversed.
\end{proof}

\subsection*{Appendix B -- Proof of Lemma \ref{lem:n3}}
In this part of the appendix we present the details we omitted in the proof of Lemma \ref{lem:n3} in Section \ref{sec4:ex4}. Lemma \ref{lem:n3a} can then be shown along the very same lines with slightly different probabilities and boundaries for the six subcases; see Fig. \ref{fig:cases2}.

\begin{proof}[Proof of Lemma \ref{lem:n3}]
%\end{proof}
We set $A=v_1 \sqrt{2}$, $B=v_2 \sqrt{2}-1$ such that for given $1/2 \leq A \leq 1$ we have $2A-1 \leq B \leq A$.

Recall that we 
	%have $|\Omega_1|= A^2/2$, $|\Omega_3| = (1-B)^2/2$ and $|\Omega_2| = 1-|\Omega_1| - |\Omega_3|$
	%and that we 
	need to evaluate \eqref{eqf_i} with the notation introduced just before this formula was stated.
The probabilities in each of the  cases below are obtained from elementary calculations.

%%%%%%%%%%%%%%%%
\paragraph{\bf Case I} The case $(x,y) \in \Omega_1=S_1$ was already discussed in the main text.
%%%%%%%%%%%%%%%%
\paragraph{\bf Case II} 
Let $(x,y) \in S_2$; i.e. $0\leq x \leq B$ and $A \leq y \leq 1$ as well as $B\leq x \leq A$ and $A \leq y \leq 1+B-x$. Then $q_3(x,y)=0$,
$$q_1(x,y)=\frac{2 \left(A x-\frac{x^2}{2}\right)}{A^2} \ \ \ \text{ and } \ \ \ 
q_2(x,y)=\frac{x (y-A)+\frac{x^2}{2}}{-\frac{A^2}{2}-\frac{1}{2} (1-B)^2+1}.$$ 
Hence, we get
\begin{align*}
f_2&(x,y)= \frac{1}{9 A^2 \left(A^2+(B-2)
   B-1\right)}  \times\\
   &  \big( x^3 \left(-2 y \left(-6 A^2-3 (B-2) B+1\right)-8A\right) \\
   &+x^2 \left(3 A^2 y^2 \left(3 A^2+3 (B-2) B+1\right)-4 A y \left(6 A^2+3 (B-2) B-1\right)+6 A^2+(2-B) B+1\right) \\
   &+x \left(4 A^3-2 A^2 y+2 A (B-2) B-2 A\right) +2 x^4b\big).
   \end{align*}
%and
%$$ \int_{S_{II}} f_2(x,y) \dd y \dd x =. $$

%%%%%%%%%%%%%%%%
\paragraph{\bf Case III} 
Let $(x,y) \in S_{3}$; i.e. $0 \leq x \leq A/2$ and $A-x \leq y \leq A$ as well as $A/2 \leq x \leq A$ and $x \leq y \leq A$. Then $q_3(x,y)=0$ and 
$$q_1(x,y)=\frac{2 \left(x y-\frac{1}{2} (-A+x+y)^2\right)}{A^2} \ \ \ \text{ and } \ \ \ 
q_2(x,y)=-\frac{(-A+x+y)^2}{A^2+(B-2) B-1}.$$ 
Hence, we get
\begin{align*}
f_3&(x,y)= \frac{1}{9 A^2 \left(A^2+(B-2) B-1\right)} \times \\
&\left( 2 x^4 + x^3 \left(12 A^2 y-8 A+6 (B-2) B y-2
   y\right) \right. \\
 &+x \left(12 A^4 y-24
   A^3 y^2-4 A^3+6 A^2 B^2 y-12 A^2 B y+12 A^2 y^3+12A^2 y-12 A (B-2) B y^2 \right . \\
  & \qquad \left. +2 A (B-2) B-4 A y^2-2 A+6 (B-2) B y^3-2 y^3\right) \\
  & +x^2 \left(9 A^4 y^2-24 A^3 y+3 A^2 (3 (B-2) B+1) y^2+10 A^2-12 A (B-2) B y \right. \\
  &\qquad \left. -4 A y+(2-B) B+4 y^2+1\right) \\
  &  +2 A (B-2) B y-8 A y^3-2 A  y-(B-2) B y^2+2 y^4+y^2 +2 A^2 B \\
  &\qquad \left. +10 A^2 y^2+A^2 -4 A^3 y-A^2 B^2 \right). 
\end{align*}
%and
%$$ \int_0^{A} \int_{0}^{A-x} \frac{x y(2 + 3 (-4 + 3A^2)x y)}{9 A^2} \dd y \dd x = \frac{1}{540} A^2 (5-4A^2 + 3A^4). $$

%%%%%%%%%%%%%%%%
\paragraph{\bf Case IV} 
Let $(x,y) \in S_{4}$; i.e. $A \leq x \leq 1+B-A$ and $A \leq y \leq 1+B-x$. Then $q_1(x,y)=1, q_3(x,y)=0$ and
$$q_2(x,y)= \frac{x y-\frac{x^2}{2}}{-\frac{A^2}{2}-\frac{1}{2} (1-B)^2+1}.$$ 
Hence, we get
\begin{align*}
f_4(x,y)= \frac{(-1 - 2 B (1 - 3 x y)^2 + B^2 (1 - 3 x y)^2 +  x^2 (4 + 3 y (-4 x + y + 3 A^2 y)))}{9 (-1 + A^2 + (-2 + B) B)}.  \end{align*}
%and
%$$ \int_0^{A} \int_{0}^{A-x} \frac{x y(2 + 3 (-4 + 3A^2)x y)}{9 A^2} \dd y \dd x = \frac{1}{540} A^2 (5-4A^2 + 3A^4). $$

%%%%%%%%%%%%%%%%
\paragraph{\bf Case V} 
Let $(x,y) \in S_5$; i.e. $B \leq x \leq A$ and $1+B-x \leq y \leq 1$. Then $q_2(x,y)=q_3(x,y)=0$ and $q_1(x,y)=2xy/A^2$. Hence, we get
\begin{align*}
f_5&(x,y)= \frac{1}{9 A^2 (-1 + B)^2 (-1 + A^2 + (-2 + B) B)} \times \\
&\left( A^4 B^2-2 A^4 B y+2 A^4 B+A^4 y^2-2 A^4 y+A^4+2 A^2 (B+1)^2 \left(2 B^2+1\right)
 \right. \\
& \qquad -8 A^2 B y^3+2 A^2 (B (7 B+10)+6) y^2-4 A^2 B (B (3 B+5)+4) y+2 A^2 y^4-8 A^2
   y^3-8 A^2 y \\
   & +x^4 \left(-2 A^2-2 B^2+4 B+2\right) \\
&   +x^3 \left(-6 A^4 y+8 A^3+8 A^2 y-8 A+6
   B^4 y+B^3 (8-24 y)-2 B^2 (4-10 y)+8 B (y-2)-2
   y\right) \\
&   +x^2 \left(9 A^4 (B-2) B y^2+12 A^4 B y-3
   A^4 y^2+12 A^4 y+A^4-24 A^3 (B-2) B y-16 A^3 B \right. \\
   & \qquad  -8 A^3 y-16 A^3+2 A^2 \left(8 B^2+7\right)+3 A^2 ((B-2)
   B-1) (3 (B-2) B-1) y^2 \\
  &\qquad  +8 A^2 (B+1) (3 (B-2) B-1)y-12 A B^4 y+8 A B^3 (6 y-2)+8 A B^2 (2-5 y) \\
  &\qquad \left. -16 A B
   (y-2)+4 A y-5 B^4+8 B^3 y+4 B^3-4 B^2 y (y+2)+8
   B^2+8 B (y-1)^2+1\right) \\
&   +x \left(-6 A^4 B^2 y+12 A^4
   B y^2-12 A^4 B y-2 A^4 B-6 A^4 y^3+12 A^4 y^2-4 A^4
   y-2 A^4+12 A^3 B^2  \right. \\
 &\qquad  -16 A^3 B y+8 A^3 B+8 A^3 y^2-16
   A^3 y+12 A^3-2 A^2 B \left(6 B^3-29 B-30\right) y \\
&\qquad -12 A^2 (B-2) B y^3 +8 A^2 (B+1) (3 (B-2) B-2) y^2-4 A^2
   (B+1) (B (3 B+2)+2) \\
  & \qquad +4 A^2 y^3+18 A^2 y+10 A B^4-16 A
   B^3 y-8 A B^3+8 A B^2 y (y+2)
   \\
   & \qquad \left. \left.-16 A B^2-16 A B
   (y-1)^2-2 A\right) \right).
\end{align*}

%%and
%%$$ \int_0^{A} \int_{0}^{A-x} \frac{x y(2 + 3 (-4 + 3A^2)x y)}{9 A^2} \dd y \dd x = \frac{1}{540} A^2 (5-4A^2 + 3A^4). $$

%%%%%%%%%%%%%%%%
\paragraph{\bf Case VI} 
Let $(x,y) \in S_6$; i.e. $A \leq x \leq (B+1)/2$ and $1+B-x \leq y \leq 1$ as well as $(B+1)/2 \leq x \leq 1$ and $x \leq y \leq 1$. Then $q_1(x,y)=1$ and 
$$q_2(x,y)=\frac{-\frac{A^2}{2}-\frac{1}{2} (-B+x+y-1)^2+x
   y}{-\frac{A^2}{2}-\frac{1}{2} (1-B)^2+1} \ \ \ \text{ and } \ \ \  q_3(x,y)=\frac{(-B+x+y-1)^2}{(1-B)^2}.$$ 
Hence, we get
\begin{align*}
f_6 &(x,y)= \frac{1}{{9 (B-1)^2 \left(A^2+(B-2) B-1\right)}} \times \\
&\left( A^4 B^2-2 A^4 B y+2 A^4 B+A^4 y^2-2 A^4 y+A^4 +2 A^2 (B+1)^2 \left(2 B^2+1\right)\right . \\
&\quad -8 A^2 B y^3+2 A^2 (B (7 B+10)+6) y^2-4 A^2 B (B (3 B+5)+4) y+2 A^2 y^4-8 A^2 y^3-8 A^2 y \\
&+x^4
   \left(-2 A^2-2 B^2+4 B+2\right) \\
&   +x^3 \left(-6 A^4 y+8 A^3+8 A^2 y-8 A+6
   B^4 y+B^3 (8-24 y)-2 B^2 (4-10 y)+8 B (y-2)-2
   y\right) \\
  & +x^2 \left(9 A^4 (B-2) B y^2+12 A^4 B y-3
   A^4 y^2+12 A^4 y+A^4-24 A^3 (B-2) B y-16 A^3 B \right. \\
 & \qquad  -8 A^3y-16 A^3+2 A^2 \left(8 B^2+7\right)+3 A^2 ((B-2)
   B-1) (3 (B-2) B-1) y^2\\
   &\qquad +8 A^2 (B+1) (3 (B-2) B-1)y -12 A B^4 y+8 A B^3 (6 y-2)+8 A B^2 (2-5 y) \\
   &\qquad \left. -16 A B
   (y-2)+4 A y-5 B^4+8 B^3 y+4 B^3-4 B^2 y (y+2)+8
   B^2+8 B (y-1)^2+1\right) \\
   & +x \left(-6 A^4 B^2 y+12 A^4
   B y^2-12 A^4 B y-2 A^4 B-6 A^4 y^3+12 A^4 y^2-4 A^4
   y-2 A^4+12 A^3 B^2 \right. \\
   & \qquad -16 A^3 B y+8 A^3 B+8 A^3 y^2-16
   A^3 y+12 A^3-2 A^2 B \left(6 B^3-29 B-30\right) y \\
   &\qquad -12  A^2 (B-2) B y^3+8 A^2 (B+1) (3 (B-2) B-2) y^2-4 A^2
   (B+1) (B (3 B+2)+2) \\
   &\qquad  +4 A^2 y^3+18 A^2 y+10 A B^4-16 A
   B^3 y-8 A B^3+8 A B^2 y (y+2)
   \\
   &\qquad \left. \left.
   -16 A B^2-16 A B
   (y-1)^2-2 A\right) \right).\\
\end{align*}
%and
%$$ \int_0^{A} \int_{0}^{A-x} \frac{x y(2 + 3 (-4 + 3A^2)x y)}{9 A^2} \dd y \dd x = \frac{1}{540} A^2 (5-4A^2 + 3A^4). $$

%With similar calculations as in Case I we can obtain expressions $f_2, \ldots, f_6$ for the expected value of the discrepancy function for points $(x,y)$ in each of the sets and in dependence of the parameters $A$ and $B$. 

Inserting the explicit expressions of the functions $f_1, \ldots, f_6$ into \eqref{eqf_i} 
gives us the following rational function in $A$ and $B$ which we can minimize over $1/2\leq A \leq 1$ and $2A-1 \leq B \leq A$ in order to obtain the two parameters $A$ and $B$ that generate the partition with the smallest expected discrepancy in this family: 
\begin{align*}
&\E \cL_2^2(\cP_{\bOmega_{v_1,v_2}^{(3)}} ) =
\frac{1}{25920 A^2 (A^2+(B-2) B-1)} \times \\
&\hspace{1cm} \left. (480 A^8+6912 A^7 (B+1)+16 A^6 (-527 + 330 B + 195 B^2) \right. \\
&\hspace{1cm} \left. -384 A^5 (13 + 67 B + 87 B^2 + 33 B^3)+ \right. \\
&\hspace{1cm}\left. 12 A^4 (829 + 1952 B + 2720 B^2 + 1228 B^3 + 431 B^4) \right. \\
&\hspace{1cm}\left. +48 A^3 (-54 - 181 B - 204 B^2 + 23 B^3 + 120 B^4 + 24 B^5) \right. \\
 &\hspace{1cm}  \left. -4 A^2 (118 + 104 B + 345 B^2 + 776 B^3 - 728 B^4 + 1368 B^5 + 393 B^6)\right. \\
&\hspace{1cm}   \left. +48 A (-10 + 7 B + 76 B^2 - 3 B^3 - 140 B^4 - 42 B^5 + 36 B^6 + 12 B^7)\right. \\
&\hspace{1cm}\left. -(56 B + 774 B^2 + 556 B^3 - 1364 B^4 - 1728 B^5 - 70 B^6 + 396 B^7 + 
 99 B^8)\right).
\end{align*}
This function can be minimised using a standard computer algebra system.
The minimum of this function is $0.0268044$ for the parameter values $A= 0.7512174$ and $B= 0.513013$; see Fig. \ref{N3min}. 
These parameter values correspond to $v_1 = 0.5311\ldots$ and $v_2 = 1.0698\ldots$.
\end{proof}

%\textcolor{red}{[maybe for Lemma 5, give the probabilities and the boundaries of the regions.]}

\end{document}